\PassOptionsToPackage{sort&compress}{natbib}
\documentclass[final,3p,times,10pt]{elsarticle}
\usepackage[centertags]{amsmath}
\usepackage{amsfonts,amssymb,amsthm,mathtools}
\usepackage{epstopdf,newlfont,graphicx,subfigure,booktabs}
\usepackage{bm,cancel,algorithm,algpseudocode,multirow}
\usepackage[mathscr]{euscript}
\usepackage[short,nocomma]{optidef}
\usepackage{appendix}
\usepackage[colorlinks=true]{hyperref}
\newcommand\hcancel[2][black]{\setbox0=\hbox{$#2$}%
\rlap{\raisebox{.25\ht0}{\textcolor{#1}{\rule{0.8\wd0}{0.5pt}}}}#2} 
\newcommand\hcancelt[2][black]{\setbox0=\hbox{$#2$}%
\rlap{\raisebox{.25\ht0}{\textcolor{#1}{\hspace{0.3mm}\rule{0.7\wd0}{0.75pt}}}}#2} 
\newtheorem{thm}{Theorem}[section]
\newtheorem{lem}{Lemma}[section]

\newtheorem{rem}{Remark}[section]
\newtheorem{rot}{Rule of Thumb}

\numberwithin{algorithm}{section}
\numberwithin{equation}{section}
\renewcommand{\theequation}{\thesection.\arabic{equation}}

\newcommand{\bma}{\bm{a}}
\newcommand{\bmb}{\bm{b}}
\newcommand{\bmx}{\bm{x}}
\newcommand{\bmu}{\bm{u}}
\newcommand{\bmv}{\bm{v}}
\newcommand{\bmy}{\bm{y}}

\newcommand{\bmh}{\bm{h}}
\newcommand{\bmt}{\bm{t}}
\newcommand{\bmxi}{\bm{\xi}}
\newcommand{\bmeta}{\bm{\eta}}
\newcommand{\bmK}{\bm{K}}
\newcommand{\bmX}{\bm{X}}
\newcommand{\bmT}{\bm{T}}
\newcommand{\bmE}{\bm{E}}
\newcommand{\bmP}{\bm{P}}
\newcommand{\bmW}{\bm{W}}

\newcommand{\bmvarpi}{\bm{\varpi}}

\newcommand{\hu}{\hat{u}}
\newcommand{\hFP}{\hat{\F{P}}}

\newcommand{\hCG}{\hat{\C{G}}}

\newcommand{\MB}[1]{\mathbb{#1}}

\newcommand{\MBR}{\MB{R}}
\newcommand{\MBS}{\MB{S}}
\newcommand{\MBRP}{\MBR^+}
\newcommand{\MBRzer}{\MBR_0}
\newcommand{\MBRzerP}{\MBRzer^+}
\newcommand{\MBRmh}{\oset{\rightarrow}{\MBR}_{-1/2}}
\newcommand{\MBRmhzer}{\MBR_{-1/2}^-}
\newcommand{\MBZPL}{\oset{\rightarrow}{\MBZP}}
\newcommand{\MBZ}{\mathbb{Z}}
\newcommand{\MBZP}{\MBZ^+}
\newcommand{\MBZzer}{\MBZ_0}
\newcommand{\MBZzerP}{\MBZzer^+}
\newcommand{\MBJ}{\mathbb{J}}
\newcommand{\MBN}{\mathbb{N}}
\newcommand{\MBG}{\mathbb{G}}
\newcommand{\MBTG}{\mathbb{TG}}
\newcommand{\F}[1]{\mathbf{#1}}
\newcommand{\C}[1]{\mathcal{#1}}
\newcommand{\FC}[1]{\mathbfcal{#1}}
\newcommand{\MF}[1]{\mathfrak{#1}}

\newcommand{\MFF}{\MF{F}}
\newcommand{\FOmega}{\F{\Omega}}
\newcommand{\SCR}[1]{\mathscr{#1}}
\newcommand{\foralla}{\,\forall_{\mkern-6mu a}\,}
\newcommand{\foralle}{\,\forall_{\mkern-6mu e}\,}
\newcommand{\foralls}{\,\forall_{\mkern-6mu s}\,}

\newcommand{\Def}[1]{\text{Def}\left(#1\right)}
\newcommand{\canczer}[1]{#1_{\hcancel{0}}}
\newcommand{\cancbra}[1]{\hcancel{[}#1\hcancelt{]}}

\newcommand{\tf}{\tilde{f}}

\newcommand{\tFP}{\tilde{\F{P}}}
\newcommand{\bmone}{\bm{\mathit{1}}}
\newcommand{\bmzer}{\bm{\mathit{0}}}
\newcommand{\trp}[1]{\text{trp}\left(#1\right)}
\newcommand{\siz}[1]{\text{siz}\left(#1\right)}
\newcommand{\resh}[3]{\text{resh}_{#1,#2}\left(#3\right)}
\newcommand{\reshs}[2]{\text{resh}_{#1}\left(#2\right)}
\newcommand{\vect}[1]{\text{vec}\left(#1\right)}

\newcommand{\diag}[1]{\text{diag}\left(#1\right)}
\def\simgt{\,\hbox{\lower0.6ex\hbox{$>$}\llap{\raise0.3ex\hbox{$\sim$}}}\,}
\def\simlt{\,\hbox{\lower0.6ex\hbox{$<$}\llap{\raise0.3ex\hbox{$\sim$}}}\,}
\def\simgteq{\,\hbox{\lower0.6ex\hbox{$\ge$}\llap{\raise0.6ex\hbox{$\sim$}}}\,}
\def\simlteq{\,\hbox{\lower0.6ex\hbox{$\le$}\llap{\raise0.6ex\hbox{$\sim$}}}\,}
\DeclareOldFontCommand{\rm}{\normalfont\rmfamily}{\mathrm}
\DeclareMathAlphabet\mathbfcal{OMS}{cmsy}{b}{n}
\DeclareMathOperator{\Span}{span}
\makeatletter
\def\user@resume{resume}
\def\user@intermezzo{intermezzo}
\newcounter{previousequation}
\newcounter{lastsubequation}
\newcounter{savedparentequation}
\renewenvironment{subequations}[1][]{%
      \def\user@decides{#1}%
      \setcounter{previousequation}{\value{equation}}%
      \ifx\user@decides\user@resume 
           \setcounter{equation}{\value{savedparentequation}}%
      \else  
      \ifx\user@decides\user@intermezzo
           \refstepcounter{equation}%
      \else
           \setcounter{lastsubequation}{0}%
           \refstepcounter{equation}%
      \fi\fi
      \protected@edef\theHparentequation{%
          \@ifundefined {theHequation}\theequation \theHequation}%
      \protected@edef\theparentequation{\theequation}%
      \setcounter{parentequation}{\value{equation}}%
      \ifx\user@decides\user@resume 
           \setcounter{equation}{\value{lastsubequation}}%
         \else
           \setcounter{equation}{0}%
      \fi
      \def\theequation  {\theparentequation  \alph{equation}}%
      \def\theHequation {\theHparentequation \alph{equation}}%
      \ignorespaces
}{%
  \ifx\user@decides\user@resume
       \setcounter{lastsubequation}{\value{equation}}%
       \setcounter{equation}{\value{previousequation}}%
  \else
  \ifx\user@decides\user@intermezzo
       \setcounter{equation}{\value{parentequation}}%
  \else
       \setcounter{lastsubequation}{\value{equation}}%
       \setcounter{savedparentequation}{\value{parentequation}}%
       \setcounter{equation}{\value{parentequation}}%
  \fi\fi
  \ignorespacesafterend
}

\makeatother
\makeatletter
\newcommand{\oset}[3][0ex]{%
  \mathrel{\mathop{#3}\limits^{
    \vbox to#1{\kern-2\ex@
    \hbox{$\scriptstyle#2$}\vss}}}}
\makeatother
\begin{document}
\begin{frontmatter}
\title{Direct Integral Pseudospectral and Integral Spectral Methods for Solving a Class of Infinite Horizon Optimal Output Feedback Control Problems Using Rational and Exponential Gegenbauer Polynomials}
\author[XMUM,Assiut]{Kareem T. Elgindy\corref{cor1}}
\address[XMUM]{Mathematics Department, School of Mathematics and Physics, Xiamen University Malaysia, Sepang 43900, Malaysia}
\address[Assiut]{Mathematics Department, Faculty of Science, Assiut University, Assiut 71516, Egypt}
\ead{kareem.elgindy@xmu.edu.my}
\cortext[cor1]{Corresponding author}
\author[sohag]{Hareth M. Refat}
\ead{harith\_refaat@science.sohag.edu.eg; hareth.mohamed.refat@gmail.com}
\address[sohag]{Mathematics Department, Faculty of Science, Sohag University, Sohag 82524, Egypt} 

\begin{abstract}
This study is concerned with the numerical solution of a class of infinite-horizon linear regulation problems with state equality constraints and output feedback control. We propose two numerical methods to convert the optimal control problem into nonlinear programming problems (NLPs) using collocations in a semi-infinite domain based on rational Gegenbauer (RG) and exponential Gegenbauer (EG) basis functions. We introduce new properties of these basis functions and derive their quadratures and associated truncation errors. A rigorous stability analysis of the RG and EG interpolations is also presented. The effects of various parameters on the accuracy and efficiency of the proposed methods are investigated. The performance of the developed integral spectral method is demonstrated using two benchmark test problems related to a simple model of a divert control system and the lateral dynamics of an F-16 aircraft. Comparisons of the results of the current study with available numerical solutions show that the developed numerical scheme is efficient and exhibits faster convergence rates and higher accuracy.
\end{abstract}

\end{frontmatter}
\section{Introduction}
\label{Int}
Infinite-horizon optimal control problems (IHOCs) of continuous-time linear systems with linear state equality constraints and output feedback control arise in systems with input equality constraints, particularly when the actuator dynamics are also considered in the problem so that the actuator outputs are augmented to the plant state vector. Typical examples of these systems include divert control system (DCS) models and the lateral dynamics of an F-16 aircraft. Only a few studies have appeared in the literature to solve such problems. An early attempt to solve the IHOCs of continuous-time linear systems with state equality constraints occurred in \cite{KO20071573}, where various existence conditions for constraining state feedback control were found before determining the optimal feedback gain via reduction of the control input space dimension. \citet{cha2019infinite} extended this study to cover problems with output feedback by following the same approach presented earlier in \cite{KO20071573}. However, a major drawback of the latter approach is the need to derive and solve a sophisticated system of optimal gain design equations for constrainable output feedback gains, which require either finding the optimal form of a non-unique basis matrix or fixing a basis matrix a priori before finding the optimal gain and then comparing the effect of the various basis matrices on the performance of the controller via numerical simulations. 

In this study, we present the TG-IPS and TG-IS methods: Direct integral pseudospectral (IPS) and integral spectral (IS) methods based on rational Gegenbauer (RG) and exponential Gegenbauer (EG) functions for solving IHOCs of continuous-time linear systems with linear state equality constraints and output feedback control. Both methods overcome the drawbacks of the method of \citet{KO20071573} and \citet{cha2019infinite} by directly converting the integral form of the problem into an NLP via efficient collocation on the original semi-infinite domain based on RG and EG functions, which we collectively refer to as transformed Gegenbauer (TG) functions. All integrals included in the problem in its integral form are approximated using novel and highly accurate quadratures based on TG functions. To guarantee the convergence of the proposed methods, we derive the parameter ranges over which the TG-based interpolation/collocation is stable through rigorous interpolation/collocation stability analysis. We later show that the TG-IS method is superior to the TG-IPS method in terms of computational complexity. Because the performance of both methods may vary considerably with the parameters used, a significant part of our study is devoted to deriving a crucial rule of thumb that provides a useful means to optimize the performance of the proposed methods by determining the optimal ranges of the collocation mesh size, Gegenbauer index, and the mapping scaling parameter associated with the TG basis functions. While TG functions were applied earlier in several works, cf. \cite{li2020diagonalized,hajimohammadi2020new,baharifard2022novel,parand2018numerical} to mention a few, their application for solving the problem under study has never been considered in the literature, to the best of our knowledge. Another important contribution of this study is the ability of the proposed TG-IS method to determine the approximate optimal state and control variables with exponential convergence rates directly in the original semi-infinite domain. In contrast, common early approaches for solving IHOCs aimed to solve this class of problems through transformation into finite-horizon optimal control problems by using certain parametric mappings and then collocating the latter using classical Legendre and Gegenbauer polynomials. While these approaches can converge exponentially to near-optimal approximations for a coarse collocation mesh grid size, they usually diverge as the number of collocation points grows large if the computations are carried out using floating-point arithmetic, and the discretization uses a single mesh grid, regardless of whether they are of the Gauss/Gauss-Radau type or equally spaced; cf. \cite{elgindy2023direct} and the references therein. The current methods do not suffer from this limitation, as they generally remain stable for large mesh grids and certain parameter ranges, as we will demonstrate later in Section \ref{sec:NS1} and \ref{sec:STGI1}. For a comprehensive review of IPS and IS methods, the reader may consult \cite{elgindy2023direct,dahy2022high,Elgindy2016,elgindy2018high,
dai2016integral,greengard1991spectral,driscoll2010automatic}, and the references therein.

The remainder of this paper is organized as follows. In Section \ref{sec:PN1}, we present some mathematical notations to be used later in this study. The problem under study is described in Section \ref{sec:PS1}. The proposed TG-IPS and TG-IS methods are presented in Sections \ref{sec:NOIHOC} and \ref{sec:NOIHOC2}. Section \ref{sec:TQTE1} provides an error analysis of the TG quadrature derived in Section \ref{sec:NOIHOC}. Section \ref{sec:NS1} presents extensive numerical simulations to assess the performance of the TG quadratures and the proposed TG-IS method. Section \ref{sec:Conc} provides some concluding remarks followed by a future work in Section \ref{sec:FW1}. \ref{sec:TOTIHOCP1} provides some properties of RG and EG functions and their associated Gaussian quadratures. Finally, \ref{sec:STGI1} presents a rigorous stability study of TG interpolation/collocation. 


\section{Preliminary Notations}
\label{sec:PN1}
The following notations are used throughout this study to simplify the mathematical formulas. Most of these notations appeared recently in \cite{elgindy2023new,elgindy2023fouriera,elgindy2023fourierb}; however, we present them here together with some new notations to keep the paper as self-explanatory as possible.

\noindent\textbf{Logical and Relational Symbols.} The  symbols $\forall, \foralla, \foralle$, and $\foralls$ stand for the phrases ``for all,'' ``for any,'' ``for each,'' and ``for some,'' in respective order. $\simlt$ and $\simlteq$ mean the statements ``asymptotically less than'' and ``asymptotically less than or equal to,'' respectively.\\[0.5em] 
\noindent \textbf{Set and List Notations.} $\MBR, \canczer{\MBR}$, and $\MBRzerP$ are the sets of real, non-zero real, and nonnegative real numbers, respectively. $\MBZ, \MBZzerP$, and $\MBZPL$ are the sets of integer, nonnegative integer, and sufficiently large positive integer numbers, respectively. We define also the two sets $\MBRmh = \{x \in \MBR: x > -1/2\}$ and $\MBRmhzer = \{x \in \MBR: -1/2 < x < 0\}$. $\MFF$ denotes the set of all real-valued functions. $\Span \MBS$ stands for the set of all linear combinations of the vectors/functions in the set $\MBS$. The notations $i$:$j$:$k$ or $i(j)k$ indicate a list of numbers from $i$ to $k$ with increment $j$ between numbers, unless the increment equals one where we use the simplified notation $i$:$k$. The list of symbols $y_1, y_2, \ldots, y_n$ is denoted by $\left. y_i \right|_{i=1:n}$ or simply $y_{1:n}$, and their set is represented by $\{y_{1:n}\}\,\foralla n \in \MBZP$. This notation is further extended to any symbol with multiple subscripts; for example, $y_{1:m,0:n}$ stands for the list of symbols $y_{1,0}, y_{1,1}, \ldots, y_{1,n}, y_{2,0}, \ldots, y_{2,n}, \ldots, y_{m,n}$. We define $\MBJ_n = \{0$:$n\}$ and $\MBN_N = \{1$:$N\}\,\foralla n \in \MBZzerP, N \in \MBZP$. $\MBG_n^{\alpha} = \left\{x_{n,0:n}^{\alpha}\right\}$ is the zeros set of the $(n+1)$st-degree Gegenbauer polynomial with index $\alpha \in\,\MBRmh$ (aka the Gegenbauer-Gauss (GG) points set) $\foralla n \in \MBZP$. Finally, the specific interval $[0, c]$ is denoted by $\F{\Omega}_c\,\forall c > 0$. For example, $[0, x_{n,j}^{\alpha}]$ is denoted by ${\F{\Omega}_{x_{n,j}^{\alpha}}}$; moreover, ${\F{\Omega}_{x_{n,0:n}^{\alpha}}}$ stands for the list of intervals ${\F{\Omega}_{x_{n,0}^{\alpha}}}, {\F{\Omega}_{x_{n,1}^{\alpha}}}, \ldots, {\F{\Omega}_{x_{n,n}^{\alpha}}}$.\\[0.5em] 
\textbf{Function Notations.} $\delta_{m,n}$ is the usual Kronecker delta function of variables $m$ and $n$. $\Gamma$ and $\digamma$ denote the Gamma function and its logarithmic derivative. $H_x$ is the Harmonic number real-continuation function defined by $H_x = \digamma(x) + \gamma_{em}\,\forall x \in \MBR: H_n$ is the usual Harmonic number function $\forall n \in \MBZP$ and $\gamma_{em}$ is the Euler-Mascheroni constant defined by 
\[\gamma_{em} = \mathop {\lim }\limits_{n \to \infty } \left( { -\ln n + \sum\limits_{k = 1}^n {\frac{1}{k}} } \right) \approx 0.577216,\]
rounded to six decimal digits. $G_{n}^{\alpha}(x)$ is the $n$th-degree Gegenbauer polynomial with index $\alpha \in\,\MBRmh$ such that $G_{n}^{\alpha}(1) = 1\,\forall n \in \MBZzerP$. For convenience, we shall denote $h\left(x_{n,j}^{\alpha}\right)$ and $h\left({}_i^Lt_{n,j}^{\alpha}\right)$ by $h_{n,j}^{\alpha}$ and $h_{i,n,j}^{\alpha,L}\,\foralla h \in \MFF, {}_i^Lt_n^{\alpha} \in \MBRzerP: (n,j,i,L) \in \MBZzerP \times \MBJ_n \times \MBN_2 \times \MBZP$, unless stated otherwise.\\[0.5em]
\textbf{Integral Notations.} We denote $\int_0^{b} {h(t)\,dt}$ and $\int_a^{b} {h(t)\,dt}$ by $\C{I}_{b}^{(t)}h$ and $\C{I}_{a, b}^{(t)}h$, respectively, $\foralla h \in \MFF, \{a, b\} \subset \MBR$. $\C{I}_{\FOmega}^{(x)} {h}$ means $\int_{\FOmega} {h(x)\,dx}\,\foralla$ interval $\FOmega \subseteq \MBR$. If the integrand function $h$ is to be evaluated at any other expression of $x$, say $u(x)$, we express $\int_0^{b} {h(u(x))\,dx}$ and $\int_a^b {h(u(x))\,dx}$ with a stroke through the square brackets as $\C{I}_{b}^{(x)}h\cancbra{u(x)}$ and $\C{I}_{a,b}^{(x)}h\cancbra{u(x)}$ in respective order. Here, the strike-through notation is introduced to indicate that $h$ is a function of $u(x)$ and the integration is to be evaluated under this assumption; this is useful if we wish to avoid evaluating the integral of $h$ first and then multiply the result by $u(x)$.\\[0.5em]
\textbf{Space, Norm, and Inner Product Notations.} $\Def{\FOmega}$ is the space of functions defined on the set $\FOmega$. $\foralla \FOmega \subseteq \MBR, L^p({\FOmega})$ is the Banach space of measurable functions $u \in \Def{\FOmega}$ with the norm ${\left\| u \right\|_{{L^p_w(\FOmega)}}} = {\left[\C{I}_{\FOmega}^{(x)}{\left({|u|}^p\,w\right)}\right]^{1/p}} < \infty\,\forall p \ge 1$, where $w$ is some weight function in the usual sense; the subscript $w$ in ${L^p_w(\FOmega)}$ is omitted when $w(x) = 1\,\forall x \in \FOmega$. For $p = \infty$, the space ${L^{\infty}(\FOmega)}$ is the space of bounded measurable functions with the norm $\left\|u\right\|_{L^{\infty}(\FOmega)} = \sup_{\bmx \in \FOmega} {\left|u(\bmx)\right|}\,\foralla \FOmega \subseteq \MBR^n$. We denote by $(u, v)_w$ the inner product of the space $L^2_w(\FOmega)$ such that $(u,v)_w = \C{I}^{(x)}_{\FOmega} {(u v w)}\,\foralla \{u,v\} \subset \Def{\FOmega}$. We define the Sobolev space $H^m_w(\FOmega) = \{u: u^{(k)} \in L_w^2(\FOmega)\,\forall k \in \MBJ_m\}$. Finally, $\left\|\C{E}\right\|_{\max}$ denotes the largest element of $\C{E}$ in magnitude $\foralla n$-dimensional array of real numbers $\C{E}$.\\[0.5em]
\textbf{Vector Notations.} $\forall \{i,n,L\} \subset \MBZP$, we shall use the shorthand notations $\bmx_n^{\alpha} \left(\text{or }{x_{0:n}^{\alpha}}^{\hspace{-2mm}\top}\right)$ and ${}_i^L\bmt_n^{\alpha} \left(\text{or }{{}_i^Lt_{0:n}^{\alpha}}^{\hspace{-2mm}\top}\right)$ to stand for the column vectors $[x_{n,0}^{\alpha}, x_{n,1}^{\alpha}, \ldots, x_{n,n}^{\alpha}]^{\top}$ and $[{}_i^Lt_{n,0}^{\alpha}, {}_i^Lt_{n,1}^{\alpha}, \ldots, {}_i^Lt_{n,n}^{\alpha}]^{\top}$ in respective order. Moreover, $h_{n,0:n}^{\alpha}, h_{i,n,0:n}^{\alpha,L}$, and $c^{0:n}$ denote $[h_{n,0}^{\alpha}, h_{n,1}^{\alpha}, \ldots, h_{n,n}^{\alpha}]^{\top}, [h_{i,n,0}^{\alpha,L}, h_{i,n,1}^{\alpha,L}, \ldots, h_{i,n,n}^{\alpha,L}]^{\top}$, and the $n$th-dimensional row vector $[c^0, c^1, \ldots, c^n]$ $\foralla c \in \canczer{\MBR}$, respectively. In general, $\foralla h \in \MFF$ and vector $\bmy$ whose $i$th-element is $y_i \in \MBR$, the notation $h(\bmy)$ stands for a vector of the same size and structure of $\bmy$ such that $h(y_i)$ is the $i$th element of $h(\bmy)$. Moreover, by $\bmh(\bmy)$ (or $h_{1:m}\cancbra{\bmy}$) with a stroke through the square brackets, we mean $[h_1(\bmy), \ldots, h_m(\bmy)]^{\top}\,\foralla m$-dimensional column vector function $\bmh$, with the realization that the definition of each array $h_i(\bmy)$ follows the former notation rule $\foralle i$. We adopt the notation $\C{I}_{{\bmt_{n}}}^{(t)}h$ to denote the $(n+1)$st-dimensional column vector $\left[ {\C{I}_{{t_{n,0}}}^{(t)}h,\C{I}_{{t_{n,1}}}^{(t)}h, \ldots ,\C{I}_{{t_{n,n}}}^{(t)}h} \right]^{\top}$. Furthermore, we write $\C{I}_{{\bmt_n}}^{(t)}\bmh$ to denote the $(n+1) \times m$ matrix $\left[ \C{I}_{{t_{n,0}}}^{(t)}\bmh,\C{I}_{{t_{n,1}}}^{(t)}\bmh, \ldots ,\right.$ $\left.\C{I}_{{t_{n,n}}}^{(t)}\bmh \right]^{\top}\,\foralla m$-dimensional vector function $\bmh$.\\[0.5em] 
\textbf{Matrix Notations and Operations.} $\F{O}_n, \F{1}_n$, and $\F{I}_n$ stand for the zero, all ones, and the identity matrices of size $n$. By $[\F{A}\,;\F{B}]$ we mean the usual vertical matrix concatenation of $\F{A}$ and $\F{B}\,\foralla$ two matrices $\F{A}$ and $\F{B}$ having the same number of columns. For a two-dimensional matrix $\F{C}$, the notation $\F{C}_{\hcancel{0}}$ stands for the matrix obtained by deleting the zeroth-row of $\F{C}$. 
$\F{C}_{n,m}$ indicates that $\F{C}$ is a rectangular matrix of size $n \times m$. The notations $\F{A}^{\top}$ (or $\trp{\F{A}}$), $\vect{\F{A}}$, and $\siz{\F{A}}$ denote the transpose, the vectorization, and the size of a matrix $\F{A}$, respectively. $\max(\F{A})$ denotes the row vector containing the columns' maximum values of a matrix $\F{A}$. $\resh{m}{n}{\F{A}}$ and $\reshs{n}{\F{A}}$ are the matrices obtained by reshaping $\F{A}$ into an $m$-by-$n$ matrix and a square matrix of size $n$, respectively while preserving their column-wise ordering from $\F{A}$. $\diag{\bmv}$ denotes a square diagonal matrix with the elements of vector $\bmv$ on the main diagonal. $\odot, \oslash$, and $\otimes$ denote the usual Hadamard product and division, and Kronecker product of matrices in respective order. We adopt the notation $\F{A}_{(r)}$ to denote the $r$-times Hadamard product $\F{A} \circ \F{A} \circ  \ldots  \circ \F{A}\,\foralla$ array $\F{A}$. We assume that the Hadamard arithmetic operations have a lower order of precedence over the conventional matrix arithmetic operations. For convenience, a vector is represented in print by a bold italicized symbol while a two-dimensional matrix is represented by a bold symbol, except for a row vector whose elements form a certain row of a matrix where we represent it in bold symbol as stated earlier. For example, $\bmone_n$ and $\bmzer_n$ denote the $n$-dimensional all ones- and zeros- column vectors, while $\F{1}_n$ and $\F{O}_n$ denote the all ones- and zeros- matrices of size $n$, respectively. Finally, the symbols $\bmu^{\div}$ and $\bmv^{\div}$ stand for $\bmone_{n}^{\top}\oslash \bmu$ and $\bmone_{n}\oslash \bmv\,\forall \{\bmu^{\top},\bmv\} \subset \MBR^n$.

\section{The Problem Statement}
\label{sec:PS1}
Consider the following time-invariant linear control system of ordinary differential equations described by
\begin{subequations}
\begin{align}
\dot{\bmx}(t)&= \F{A}\bmx(t)+\F{B}\,\bmu(t),\label{eq:1}\\ 
\bmy(t)&=\F{C}\bmx(t),\label{eq:2}
\end{align}
with design state constraints
\begin{equation}\label{eq:3}
\F{D}^{\top} \bmx(t)= \bm{\mathit{0}},
\end{equation}
$\forall t \in \MBRzerP$, where $\bmx(t) = x_{1:n_x}\cancbra{t}, \bmu(t) = u_{1:n_u}\cancbra{t}, \bmy(t) = y_{1:n_y}\cancbra{t}, \F{A} \in \MBR^{n_{x} \times n_{x}}, \F{B} \in \MBR^{n_{x} \times n_{u}}, \F{C} \in \MBR^{n_{y} \times n_{x}}$, and $\F{D} \in \MBR^{n_{x} \times c_1}\,\foralls \{n_x,n_u,n_y,c_1\} \subset \MBZP$. The problem is to find the optimal output feedback control law
\begin{equation}\label{eq:OOFCL1}
\bmu(t) = -\F{K} \bmy(t)\quad \foralls \F{K} \in \MBR^{n_u \times n_y},
\end{equation}
the corresponding optimal state trajectory $\bmx$, and the output feedback $\bmy$ on the semi-infinite domain $\MBRzerP$ that satisfy Eqs. \eqref{eq:1}--\eqref{eq:3} while minimizing the functional 
\begin{equation}\label{eq:4}
J = \frac{1}{2} \C{I}_{\MBRzerP}^{(t)} {\left(\bmx^{\top} \F{Q}\,\bmx + \bmu^{\top} \F{R}\,\bmu\right)},
\end{equation}
\end{subequations}
where $\F{Q} \in \MBR^{n_x \times n_x}$ and $\F{R} \in \MBR^{n_u \times n_u}$ are symmetric positive semi-definite matrices. We refer to the OC problem described by Eqs. \eqref{eq:1}--\eqref{eq:4} by Problem $\SCR{A}$.

\section{The TG-IPS Method}
\label{sec:NOIHOC}
In this section, we describe an IPS method for solving Problem $\SCR{A}$ based on TG collocation at the TGG points. We show later how to boost the performance of this method in Section \ref{sec:NOIHOC2}. We highly recommend reading \ref{sec:TOTIHOCP1} and \ref{sec:STGI1} first before reading this section to learn more about the properties of the basis functions and their associated Gaussian quadratures employed in our methods, in addition to the stability characteristics associated with TG interpolation/collocation, which provide useful insight on the parameters ranges of values suitable to run our methods while maintaining highly accurate, robust, and efficient solutions to the mathematical problem. 

To approximate the system dynamics equations and constraints \eqref{eq:1}--
\eqref{eq:3} using TG-based collocation, we first show how to calculate the quadrature of a function $f \in \Def{\MBRzerP}$ on the successive integrals $\FOmega_{{}_i^Lt_{n,0:n}^{\alpha}}$ using the TGG quadrature. To this end, let ${}_i^L\bmE_n^{\alpha} = e^{-{}_i^L\bmt^{\alpha}_n}$ and ${}_i{\bmW^{\alpha,L}_n} = \bm{\varpi}_n^{\alpha} \oslash w^{\alpha,L}_{i}\left({}_i^L{\bmt_{n}^{\alpha}}^{\top}\right)\,\forall i  \in \MBN_2$. Using the following change of variables 
\begin{equation}
t = {}_i^Lt^{\alpha}_{n,j}\,e^{-z},
\end{equation}
we can write
\begin{equation}\label{eq:expquadrature1}
\C{I}_{{}_i^Lt_{n,j}^{\alpha}}^{(t)} {f} = {}_i^Lt_{n,j}^{\alpha}\,\C{I}_{\MBRzerP}^{(z)} {\left(e^{-z} f\cancbra{{}_i^Lt^{\alpha}_{n,j}\,e^{-z}}\right)}\quad \forall j \in \MBJ_n,
\end{equation}
which can be approximated by Lemma \ref{lem:1} as follows:
\begin{equation}\label{eq:expquadrature12}
\C{I}^{(t)}_{{}_i^Lt^{\alpha}_{n,j}} {f} \approx {}_i\C{Q}_{n,j}^{\alpha,L} = {}_i^Lt^{\alpha}_{n,j}  \left[{}_i^L\bmP^{\alpha}_n\,f\left({}_i^Lt^{\alpha}_{n,j}\,{}_i^L\bmE_n^{\alpha}\right) \right], 
\end{equation}
where ${}_i^L\bmP^{\alpha}_n= {}_i{\bmW^{\alpha,L}_n} \odot {}_i^L{\bmE_{n}^{\alpha}}^{\top}\,\forall i \in \MBN_2$ are the RG and EG integration vectors, respectively, which are collectively called ``the TG integration vector.'' Formula \eqref{eq:expquadrature12} provides a means to compute the quadrature of $f \in \Def{\MBRzerP}$ on the successive integrals $\FOmega_{{}_i^Lt_{n,0:n}^{\alpha}}$ using the TGG quadratures of the auxiliary functions ${}_i^Lg_{n,j}^{\alpha} \in \Def{\MBRzerP}$:
\begin{equation}
\quad {}_i^Lg_{n,j}^{\alpha}(t) = {}_i^Lt_{n,j}^{\alpha}\, \frac{e^{-t} f\left({}_i^Lt^{\alpha}_{n,j}\,e^{-t}\right)}{w_i^{\alpha,L}(t)}\quad \forall (i,n,\alpha,L) \in \MBN_2 \times \MBZzerP \times \MBRmh \times \MBRP.
\end{equation}
Formula \eqref{eq:expquadrature12} can be further rewritten in the following matrix notation:
\begin{equation}\label{eq:expquadrature2}
\C{I}^{(t)}_{{}_i^L\bmt^{\alpha}_n} {f} \approx {}_i\FC{Q}_n^{\alpha,L} = {}_i^L\bmt^{\alpha}_n  \odot \left[f\left({}_i^L\bmxi^{\alpha}_{n}\right)\, {}_i^L{\bmP^{\alpha}_n}^{\top} \right], 
\end{equation}
where ${}_i^L\bmxi^{\alpha}_n = {}_i^L\bmt^{\alpha}_n \otimes {}_i^L{\bmE_{n}^{\alpha}}^{\top}$. We refer to ${}_i\C{Q}_{n,j}^{\alpha,L}\,\forall i = 1$:$2$ by the RG and EG quadratures associated with the TGG node ${}_i^Lt_{n,j}^{\alpha}\,\foralle j \in \MBJ_n$, respectively, or collectively by ``the TG quadrature.'' 

Now, let $\F{X}=x_{1:n_x}^{\top}\cancbra{{}_i^L\bmt_n^{\alpha}}, \F{Y} = y_{1:n_y}^{\top}\cancbra{{}_i^L\bmt_n^{\alpha}}, \F{U}=u_{1:n_u}^{\top}\cancbra{{}_i^L\bmt_n^{\alpha}}$, and ${}_i^L{\hFP^{\alpha}_n} = {}_i^L\bmP^{\alpha}_{n} \otimes \F{I}_{n+1}$. To take advantage later of the well-conditioning of numerical integration operators during the collocation phase, we initially rewrite \eqref{eq:1} in its integral formulation as follows:
\begin{equation}\label{eq:10d1}
\bmx(t)=\F{A} \C{I}^{(z)}_t {\bmx}+\F{B} \C{I}^{(z)}_t {\bmu}+\bmx(0)\quad \forall t \in \MBRzerP.
\end{equation}
The TG quadrature \eqref{eq:expquadrature2} immediately yields
\begin{subequations}
\begin{align}
\C{I}_{{{}_i^L\bmt_{n}^{\alpha}}}^{(z)} (\F{A} \bmx) \approx \left(\bm{\mathit{1}}_{n_x}^{\top} \otimes {}_i^L\bmt^{\alpha}_n\right) \odot \left[\F{X}_{\bmxi}\left(\F{I}_{n_x} \otimes\, {}_i^L{\bmP_{n}^{\alpha}}^{\top} \right)\F{A}^{\top}\right],
\intertext{and}
\C{I}_{{{}_i^L\bmt_{n}^{\alpha}}}^{(z)} (\F{B} \bmu) \approx \left[\left(\bm{\mathit{1}}_{n_u}^{\top} \otimes {}_i^L\bmt^{\alpha}_n\right) \odot \F{U}_{\bmxi}\left(\F{I}_{n_u} \otimes\, {}_i^L{\bmP_{n}^{\alpha}}^{\top} \right)\right]\F{B}^{\top},
\end{align}
\end{subequations}
where $\F{X}_{\bmxi} = x_{1:n_x}^{\top}\cancbra{{}_i^L\bmxi_n^{\alpha}}$ and $\F{U}_{\bmxi} = u_{1:n_u}^{\top}\cancbra{{}_i^L\bmxi_n^{\alpha}}$. Therefore, the collocation of Eq. \eqref{eq:10d1} at the TGG nodes gives
\begin{subequations}
\begin{align}
\bmx\left({}_i^L\bm{t}^{\alpha}_{n}\right) &\approx \text{vec}\left[ \left(\bm{\mathit{1}}_{n_x}^{\top} \otimes{}_i^L\bm{t}^{\alpha}_{n}\right) \odot \F{X}_{\bmxi}\left(\F{I}_{n_x} \otimes\, {}_i^L{\bmP_{n}^{\alpha}}^{\top} \right)\F{A}^{\top}\right]+ \text{vec}\left[\left(\bm{\mathit{1}}_{n_u}^{\top} \otimes{}_i^L\bm{t}^{\alpha}_{n}\right) \odot \F{U}_{\bmxi}\left(\F{I}_{n_u} \otimes\, {}_i^L{\bmP_{n}^{\alpha}}^{\top} \right)\F{B}^{\top}\right]+\bmx\left(0\right) \otimes {\bm{\mathit{1}}}_{n+1} \nonumber\\
&= {}_i\bmT_{n,n_x}^{\alpha} \odot\left[\bm{\psi}_{A}\bmx\left({}_i^L\bmeta_n^{\alpha}\right)\right]+\F{B}_{n}\left[{}_i\bmT_{n,n_u}^{\alpha}\odot {}_i^L{\tFP_{n,n_u}^{\alpha}}\,\bmu\left({}_i^L\bmeta_n^{\alpha}\right) \right] +\bmx\left(0\right) \otimes {\bm{\mathit{1}}}_{n+1} ,\label{eq:17}
\end{align}
where ${}_i\bmT_{n,m}^{\alpha} = \bm{\mathit{1}}_{m}\otimes {}_i^L\bmt^{\alpha}_n\,\forall m \in \MBZP, \bm{\psi}_{A} = \F{A} \otimes {}_i^L{\hFP^{\alpha}_n}, {}_i^L\bmeta_n^{\alpha} = {}_i^L\bm{E}_n^{\alpha} \otimes {}_i^L\bm{t}^{\alpha}_{n}, \F{B}_{n}= \F{B} \otimes \F{I}_{n+1}$, and ${}_i^L{\tFP_{n,n_u}^{\alpha}}=\F{I}_{n_u} \otimes {}_i^L{\hFP^{\alpha}_n}$. The collocation of Eqs. \eqref{eq:2} and \eqref{eq:3} at the TGG nodes yield
\begin{align}
\F{Y} &= \F{X} \F{C}^\top,\label{eq:25}\\
\F{X} \F{D} &= \F{O}_{n+1,c_1},\label{eq:21}
\end{align}
respectively. We can readily approximate the performance index \eqref{eq:4} by using the TGG quadrature \eqref{eq:6} to obtain
\begin{equation}\label{eq:32}
J \approx J_{n} = \frac{1}{2}\left({}_i{\bmW^{\alpha,L}_n}\, g_{0:n}\right), 
\end{equation}
\end{subequations}
where $g_j = \bmx^{\top}\left({}_i^Lt_j^{\alpha}\right) \F{Q} \bmx\left({}_i^Lt_j^{\alpha}\right) + \bmu^{\top}\left({}_i^Lt_j^{\alpha}\right) \F{R} \bmu\left({}_i^Lt_j^{\alpha}\right)\,\forall j \in \MBJ_n$. Problem $\SCR{A}$ is now converted into an NLP in which the goal is to minimize the discrete cost functional \eqref{eq:32} subject to Constraints \eqref{eq:17} and \eqref{eq:21}. We can solve the NLP using modern optimization methods for the optimal states and control variables, $\bmx^*$ and $\bmu^*$, at the TGG points, and the corresponding optimal output feedback variables can then be obtained by using Eq. \eqref{eq:25}. Furthermore, the optimal output feedback gain $\F{K}^*$ can be estimated by solving the following linear system of equations
\begin{equation}\label{eq:OOFCL1nn1}
\F{K}^*\,{\F{Y}^*}^{\top} = -{\F{U}^*}^{\top},
\end{equation}
for $\F{K}^*$, where $\F{Y}^* = {y^*_{1:n_y}}^{\hspace{-3mm}\top}\;\cancbra{{}_i^L\bmt_n^{\alpha}}$ and $\F{U}^* = {u^*_{1:n_u}}^{\hspace{-3mm}\top}\;\cancbra{{}_i^L\bmt_n^{\alpha}}$. In the following section, we prescribe an alternative to the TG-IPS method that can significantly improve its computational performance in practice.

\section{The TG-IS Method}
\label{sec:NOIHOC2}
A practical difficulty in solving the NLP obtained by the TG-IPS method arises from the need to estimate the optimal state and control variables at the additional mesh grid of points ${}_i^L\bmeta_n^{\alpha}$ as well as the TGG points ${}_i^L\bmt_n^{\alpha}$ during the numerical optimization procedure. In particular, since $\siz{{}_i^L\bmeta_n^{\alpha}} = \text{siz}^2\left({}_i^L\bmt_n^{\alpha}\right)$, the derived NLP has the dimension $(n+1) (n+2) (n_x+n_u)$, which increases quadratically with the collocation mesh size. Therefore, the NLP may require a well-developed large-scale optimization method for large mesh grids. To overcome this difficulty, one approach is to solve the NLP in the TG spectral space by imposing the system dynamics equations and constraints at the TGG points using the TG modal expansions of the state and control variables in lieu of their nodal expansions and solve the converted NLP for the TG spectral coefficients. This procedure achieves a significant computational improvement by reducing the NLP dimensionality to only $(n+1) (n_x+n_u)$, which grows linearly with the collocation mesh size. We describe this computationally attractive approach in the following. 

Let $\F{a} = \left[\bma_{1};\ldots;\bma_{n_x} \right]$ and $\F{b}=\left[\bmb_{1};\ldots;\bmb_{n_u} \right]$ be the coefficients vectors of the TG collocants\footnote{By the TG collocant, we mean the solution that satisfies the given dynamics equations and constraints at the collocation points through TG-based collocation.} of the state and control variables such that $\bma_r = a_{r,0:L_x}$ and $\bmb_s = b_{s,0:L_u}\,\forall r \in \MBN_{n_x}, s \in \MBN_{n_u}$. Define also $\hCG_{i,0:m}^{\alpha,L}\cancbra{t} = \trp{\C{G}_{i,0:m}^{\alpha,L}\cancbra{t}}\,\forall m \in \MBZzerP$. Then we can write the inverse discrete TG transform of the  state and control variables at any time instance $t \in \MBRzerP$ as follows:
\begin{subequations}
\begin{align}
\bmx(t) &\approx \F{a}\, \C{G}_{i,0:L_x}^{\alpha,L}\cancbra{t},\label{spect1}
\\
\bmu(t) &\approx \F{b}\, \C{G}_{i,0:L_u}^{\alpha,L}\cancbra{t}.\label{spect12}
\end{align}
\end{subequations}
Formulas \eqref{spect1} and \eqref{spect12} immediately implies 
\begin{subequations}
\begin{align}
x_k\left({}_i^L\bmt^{\alpha}_n\right) &\approx  \hCG_{i,0:L_x}^{\alpha,L}\cancbra{{}_i^L\bmt^{\alpha}_n}\,\bma_k^{\top}\quad \forall k \in \MBN_{n_x},\\
u_s\left({}_i^L\bmt^{\alpha}_n\right) &\approx \hCG_{i,0:L_u}^{\alpha,L}\cancbra{{}_i^L\bmt^{\alpha}_n}\,\bmb_s^{\top}\quad \forall s \in \MBN_{n_u},
\end{align}
\end{subequations}
which can be written in compact form as follows:
\begin{subequations}
\begin{align}
\bmx\left({}_i^L\bmt^{\alpha}_n\right) &\approx \vect{\hCG_{i,0:L_x}^{\alpha,L}\cancbra{{}_i^L{\bmt^{\alpha}_n}^{\top}}\,\F{a}^{\top}},\label{eq:fgfg1}\\
\bmu\left({}_i^L\bmt^{\alpha}_n\right) &\approx \vect{\hCG_{i,0:L_u}^{\alpha,L}\cancbra{{}_i^L{\bmt^{\alpha}_n}^{\top}}\,\F{b}^{\top}}.\label{eq:fgfg2}
\end{align}
\end{subequations}
Substituting the above formulas into Formulas \eqref{eq:17}--\eqref{eq:32} reduces the NLP into the following form:
\begin{mini}
   {\F{a}, \F{b}}{J_{n} = \frac{1}{2}\left({}_i{\bmW^{\alpha,L}_n}\, g_{0:n}\right)}{}{}
   {\label{prob:Opt1}}{}
   \addConstraint{\vect{\hCG_{i,0:L_x}^{\alpha,L}\cancbra{{}_i^L{\bmt^{\alpha}_n}^{\top}}\,\F{a}^{\top}}}{= {}_i\bmT_{n,n_x}^{\alpha} \odot\left[\bm{\psi}_{A} \vect{\hCG_{i,0:L_x}^{\alpha,L}\cancbra{{}_i^L{\bmeta_n^{\alpha}}^{\top}}\,\F{a}^{\top}}\right]+\F{B}_{n}\left[{}_i\bmT_{n,n_u}^{\alpha}\odot {}_i^L{\tFP_{n,n_u}^{\alpha}}\,\vect{\hCG_{i,0:L_u}^{\alpha,L}\cancbra{{}_i^L{\bmeta_n^{\alpha}}^{\top}}\,\F{b}^{\top}}\right]}{}
   \addConstraint{}{+ \left[\F{a}\, \trp{(-1)^{0:L_x}}\right] \otimes {\bm{\mathit{1}}}_{n+1},}{}
   \addConstraint{\left(\hCG_{i,0:L_x}^{\alpha,L}\cancbra{{}_i^L\bmt_n^{\alpha}}\,\F{a}^{\top}\right) \F{D}}{= \F{O}_{n+1,c_1},}{}
 \end{mini} 
where $g_j = \hCG_{i,0:L_x}^{\alpha,L}\cancbra{{}_i^Lt_j^{\alpha}}\,\F{a}^{\top} \F{Q} \F{a}\,\C{G}_{i,0:L_x}^{\alpha,L}\cancbra{{}_i^Lt_j^{\alpha}} + \hCG_{i,0:L_u}^{\alpha,L}\cancbra{{}_i^Lt_j^{\alpha}}\,\F{b}^{\top} \F{R} \F{b}\,\C{G}_{i,0:L_u}^{\alpha,L}\cancbra{{}_i^Lt_j^{\alpha}}\,\forall j \in \MBJ_n$. We can solve NLP \eqref{prob:Opt1} using standard optimization software for the optimal TG spectral coefficients $\F{a}^*$ and $\F{b}^*$, and then recover the optimal state and control collocants from Eqs. \eqref{eq:fgfg1} and \eqref{eq:fgfg2}. The optimal output feedback at the TGG collocation points can later be obtained from Eq. \eqref{eq:25} and the inverse discrete TG transform \eqref{spect1} in the following form:
\begin{equation}
\F{Y}^* = \left(\hCG_{i,0:L_x}^{\alpha,L}\cancbra{{}_i^L\bmt^{\alpha}_n}\,{\F{a}^*}^{\top}\right) \F{C}^\top.
\end{equation}
Furthermore, the optimal output feedback gain $\F{K}^*$ can be estimated by solving the following linear system of equations
\begin{equation}\label{eq:OOFCL123}
{\F{K}^*} {\F{Y}^*}^{\top}\, = -\F{b}^*\,\C{G}_{i,0:L_u}^{\alpha,L}\cancbra{{}_i^L\bmt^{\alpha}_n},
\end{equation}
for $\F{K}^*$. We refer to the TG-IS method by the RG-IS and EG-IS methods for $i = 1, 2$, respectively.
 
\section{TG Quadrature Truncation Error}
\label{sec:TQTE1}
The following theorem states the truncation error of the derived TG quadrature ${}_i\C{Q}_{n,j}^{\alpha,L}$ as defined by Formula \eqref{eq:expquadrature12}.
\begin{thm}\label{thm:1}
If $f \in C^{2n+2}\left(\MBRzerP\right)\,\foralls n \in \MBZzerP$, then $\exists\,\{\xi_{1:2,0:n}\} \subset (-1,1)$ such that the truncation error of the TG quadrature rule \eqref{eq:expquadrature12} is given by
\begin{equation}\label{eq:TRE1}
{}_iE_{n,j}^{\alpha,L} = \C{I}_{{}_i^Lt_{n,j}^{\alpha}}^{(t)} {f} - {}_i\C{Q}_{n,j}^{\alpha,L} = \frac{{\pi {2^{- 4n - 2\alpha  - 1}}{}_i^Lt_{n,j}^\alpha (n + \alpha + 1)\Gamma (n + 2\alpha  + 1)}}{{(2n + 2)!\,{\Gamma ^2}(n + \alpha  + 2)}}\frac{{{d^{2n + 2}}}}{{d{x^{2n + 2}}}}{\left. {g_{i,j}^{\alpha,L}\left( {T_{i,L}^{ - 1}(x)} \right)} \right|_{x = \xi_{i,j} }}\quad \forall (i,j) \in \MBN_2 \times \MBJ_n,
\end{equation}
where 
\begin{equation}
g_{i,j}^{\alpha,L}(t) = \frac{e^{-t} f\left({}_i^Lt_{n,j}^{\alpha}\,e^{-t}\right)}{w_i^{\alpha,L}(t)}\quad \forall t \in \MBRzerP,
\end{equation}
and
\begin{equation}
T_{i,L}^{ - 1}(x) = \left\{ \begin{array}{l}
L(x + 1)/(1 - x),\quad i = 1,\\
L\ln [2/(1 - x)],\quad i = 2.
\end{array} \right.
\end{equation}
\end{thm} 
\begin{proof}
Notice that
\begin{equation}\label{eq:terrif1}
\C{I}_{{}_i^Lt_{n,j}^{\alpha}}^{(t)} {f} = {}_i^Lt_{n,j}^{\alpha} \C{I}_{\MBRzerP}^{(t)} {\left(g_{i,j}^{\alpha,L} w_i^{\alpha,L}\right)} = {}_i^Lt_{n,j}^{\alpha} \C{I}_{-1,1}^{(x)} {\left(g_{i,j}^{\alpha,L}\cancbra{T^{-1}_{i,L}(x)} w^{\alpha}\right)}\quad \forall (i,j) \in \MBN_2 \times \MBJ_n,
\end{equation}
by using the change of variables $x = T_{i,L}(t)$. Notice also that \cite[Eq. (A.1)]{elgindy2013fast} allows us to write the $(n+1)$-point GG quadrature rule in the following form: 
\begin{subequations}
\begin{gather}
I_{ - 1,1}^{(x)}\left( {f\,{w^\alpha }} \right) = \bmvarpi^{\alpha}_n f_{n,0:n}^\alpha  + E_n^G:\\
E_n^G = \frac{{\pi {2^{ - 4n - 2\alpha  - 1}}\left( {n + \alpha + 1} \right)\Gamma \left( {n + 2\alpha  + 1} \right)}}{{(2n + 2)!\,\Gamma^2 {\left( {n + \alpha  + 2} \right)}}}{f^{(2n + 2)}}(\eta)\quad \foralls \eta \in (-1,1).
\end{gather}
\end{subequations}
The proof is established by applying the above quadrature rule on Eq. \eqref{eq:terrif1}. 
\end{proof}

\section{Numerical Simulations}
\label{sec:NS1}
This section is divided into two subsections. In the first, we demonstrate the accuracy of the TG quadratures in practice. The second subsection is devoted to solving a benchmark set of two real-life applications modeled by Problem $\SCR{A}$. All computations were performed using MATLAB R2023a software installed on a personal laptop equipped with a 2.9 GHz AMD Ryzen 7 4800H CPU and 16 GB memory running on a 64-bit Windows 11 operating system. The numerical results presented in Section \ref{subsec:BP1} were obtained using the TG-IS method together with MATLAB fmincon solver that was stopped whenever
\[\left\| {{\bmX^{(k + 1)}} - {\bmX^{(k)}}} \right\|_2 < {10^{ - 12}}\quad \text{or}\quad \left\|J_n^{(k+1)} - J_n^{(k)} \right\|_2 < {10^{ - 12}},\]
where $\bmX^{(k)} = [\bmx^{(k)}; \bmu^{(k)}]$ and $J_n^{(k)}$ denote the concatenated vector of approximate NLP minimizers and optimal cost function value at the $k$th iteration, respectively. All computational results were generated using TG collocations with $L_x = L_u = n+1$. We used the maximum absolute discrete feasibility error at the collocation points, denoted by $\hat{\bmE} = (\hat E_i)_{i \in \MBJ_n}$, and given by $\hat{\bmE} = \max\left(\text{resh}^{\top}_{n+1,n_x}(\bmE)\right)$:
\begin{align}
\bmE &= \left|{}_i\bmT_{n,n_x}^{\alpha} \odot\left[\bm{\psi}_{A} \vect{\hCG_{i,0:L_x}^{\alpha,L}\cancbra{{}_i^L{\bmeta_n^{\alpha}}^{\top}}\,\F{a}^{\top}}\right]+\F{B}_{n}\left[{}_i\bmT_{n,n_u}^{\alpha}\odot {}_i^L{\tFP_{n,n_u}^{\alpha}}\,\vect{\hCG_{i,0:L_u}^{\alpha,L}\cancbra{{}_i^L{\bmeta_n^{\alpha}}^{\top}}\,\F{b}^{\top}}\right] + \left[\F{a}\, \trp{(-1)^{0:L_x}}\right] \otimes {\bm{\mathit{1}}}_{n+1}\right.\nonumber\\
&\left.- \vect{\hCG_{i,0:L_x}^{\alpha,L}\cancbra{{}_i^L{\bmt^{\alpha}_n}^{\top}}\,\F{a}^{\top}}\right|,\label{eq:MADFE1}
\end{align}
as one of the practical measures to assess the quality of the approximations. Notice that $\hat{\bmE}$ represents the vector of maximum absolute errors (MAEs) in approximating the discrete weak dynamical system equations of NLP \eqref{prob:Opt1} at the collocation points ${}_it_{n,0:n}^{\alpha}$.

\subsection{Practical Performance of the TG quadratures for Smooth Functions}
\label{subsec:PPTQSF1}
Consider the following three integrals:
\begin{equation}
I_1 = \C{I}_{{{}_i^Lt_{n,j}^{\alpha}}}^{(t)} {{e^{ - t}}},\quad I_2 = \C{I}_{{{}_i^Lt_{n,j}^{\alpha}}}^{(t)} {\frac{1}{t^2+1}},\quad I_3 = \C{I}_{{{}_i^Lt_{n,j}^{\alpha}}}^{(t)} {\tan^{-1}}\quad \forall (i,j) \in \MBN_2 \times \MBJ_n.
\end{equation}
All integrand functions are chosen to be sufficiently smooth and bounded on $\MBRzerP$ to test whether the TG quadrature can provide exponential convergence rates, as expected with common GG quadratures based on standard Gegenbauer polynomials. Figures \ref{fig:Fig1_M0p2}--\ref{fig:Fig2_0p5} show the maximum logarithmic error (MLE) associated with the approximation of each integral for several ranges of the parameters $\alpha, L$, and $n$, where $\C{E}_{1:2}$ refers to the errors associated with the RG and EG quadratures, respectively. The accuracy of the TG quadrature increases as $n$ increases in all cases, and the numerical errors generally decay algebraically as $n$ increases but exhibit exponential convergence rates for $\alpha = 0.5$ within a certain ``optimal'' range of $L$ values. Notice from the figures that the TG quadrature errors decrease exponentially as $L \to 0$ in all cases. In addition, as $L$ moves away from $0$, the RG quadrature errors generally rise and fall several times with two discerned patterns of global minimum values: (i) When the RG quadrature converges algebraically, the error global minimum value (EGMV) occurs when $L \in (10^{-2},1) \cup (3, 6]$ and (ii) when the RG quadrature converges exponentially, the EGMV falls within the range $15 < L < 25$ with an approximate global minimum at $L = 20$ based on the experimental data. The profile of the EG quadrature errors is slightly different from that of the RG quadrature as $L$ moves away from $0$ in the sense that: (i) When the EG quadrature converges algebraically, the EGMV occurs only when $L \in (10^{-2},1)$, and (ii) when the EG quadrature converges exponentially, the EGMV falls within the range $10 < L < 20$ with an approximate global minimum at $L = 15$ based on the experimental data. Figures \ref{fig:Fig3n20LS} and \ref{fig:Fig3n20LL} give more insight into the error behavior relative to the parameter $\alpha$, where we can observe two important remarks: (i) When $L \to  0$, the quadrature errors are generally strictly decreasing as $\alpha$ decreases, so that the error approaches its minimum value as $\alpha \to -0.5$; (ii) when $L$ moves away from zero, the MLE profile exhibits a spike-like surface with a sharp spike located exactly at $\alpha = 0.5$ indicating superior accuracy at this value among all other possible choices. 

Assuming that both the optimal state and control variables are sufficiently smooth and bounded, the following rule of thumb sheds light on the optimal ordered pair of parameters $(\alpha,L)$ used by the TG-IS method for solving Problem $\SCR{A}$.

\begin{rot}\label{rot:1}
The best possible approximations to the solutions of Problem $\SCR{A}$ obtained by using the TG-IS method are most likely achieved when the TG collocation phase is performed at the TGG points set $\MBTG_{i,n}^{\alpha,L}$ for 
\begin{equation}\label{eq:BestR1}
(\alpha ,L) \in \left\{ \begin{array}{l}
\left\{ {1/2} \right\} \times (15,25),\quad \text{if }i = 1,\\
\left\{ {1/2} \right\} \times (10,20),\quad \text{if }i = 2,
\end{array} \right.
\end{equation}
if the TG collocation is performed using a stretching mapping scaling parameter ($L > 1$) since the TG quadrature generally converges exponentially in that range and the TG collocation is asymptotically stable for $n \in\,\MBZPL$. In other words, the TG-IS method should be carried out using TL-based collocation at the TL-Gauss (TLG) points set $\MBTG_{i,n}^{1/2,L}$ for $L$ values as endorsed by Formula \eqref{eq:BestR1}. On the other hand, the range of the parameters 
\item \begin{equation}\label{eq:BestR2}
(\alpha ,L) \in \left\{ \begin{array}{l}
\MBRmh \times (0,1),\quad \text{if }n\text{ is relatively small},\\
\{0\} \times (0,1),\quad \text{if }n\text{ is relatively large},
\end{array} \right.
\end{equation}
should be used $\forall i \in \MBN_2$, if the TG collocation is performed using a contracting mapping scaling parameter ($L < 1$), since the TG quadrature truncation error is generally small for $L \in (0,1)$ with faster decay rates for decreasing $\alpha$ values, and the TG interpolation associated with $\alpha \in\,\MBRmh$ is only stable for relatively small mesh grids but remains asymptotically stable for $\alpha \in \MBRzerP$. In other words, the TG-IS method should be carried out using $(\alpha,L) \in\,\MBRmh \times (0,1)$ for relatively small $n$ values, but should be performed using TC-based collocation at the TC-Gauss (TCG) points set $\MBTG_{i,n}^{0,L}$ for $L \in (0,1)$ and relatively large $n$ values, as prescribed by Formula \eqref{eq:BestR2}. 
\end{rot}

\begin{figure}[H]
\centering
\includegraphics[scale=0.4]{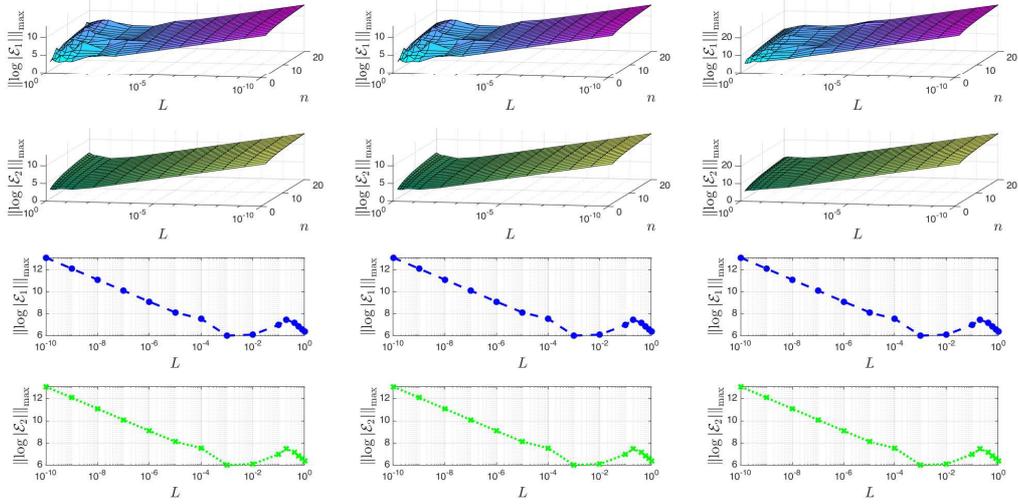}
\caption{First row: The maximum logarithmic errors in evaluating the integrals $I_1$ (left), $I_2$ (middle), and $I_3$ (right) using the RG quadrature with $\alpha = -0.2$ against $N$ and $L$ for $N = 2$:$20$ and $L = 10^{-10:1:-1},0.2$:$0.2$:$1$. Second row: The corresponding errors using the EG quadrature. Third and fourth rows: The corresponding cross-sections of the maximum logarithmic errors that are shown in the first two rows with the plane $n = 20$.}
\label{fig:Fig1_M0p2}
\end{figure}

\begin{figure}[H]
\centering
\includegraphics[scale=0.4]{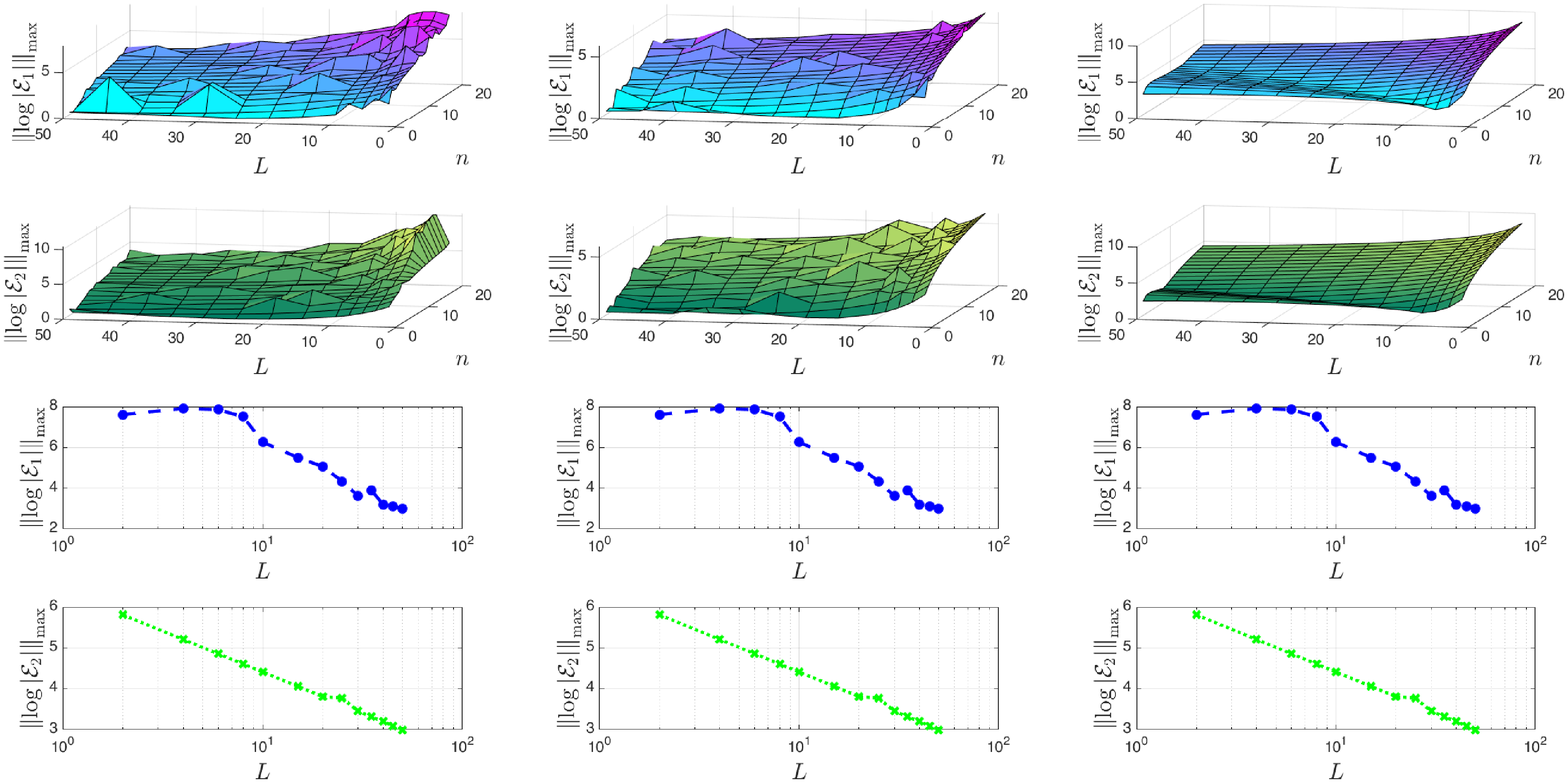}
\caption{First row: The maximum logarithmic errors in evaluating the integrals $I_1$ (left), $I_2$ (middle), and $I_3$ (right) using the RG quadrature with $\alpha = -0.2$ against $N$ and $L$ for $N = 2$:$20$ and $L = 2$:$2$:$10, 15$:$5$:$50$. Second row: The corresponding errors using the EG quadrature. Third and fourth rows: The corresponding cross-sections of the maximum logarithmic errors that are shown in the first two rows with the plane $n = 20$.}
\label{fig:Fig2_M0p2}
\end{figure}

\begin{figure}[H]
\centering
\includegraphics[scale=0.4]{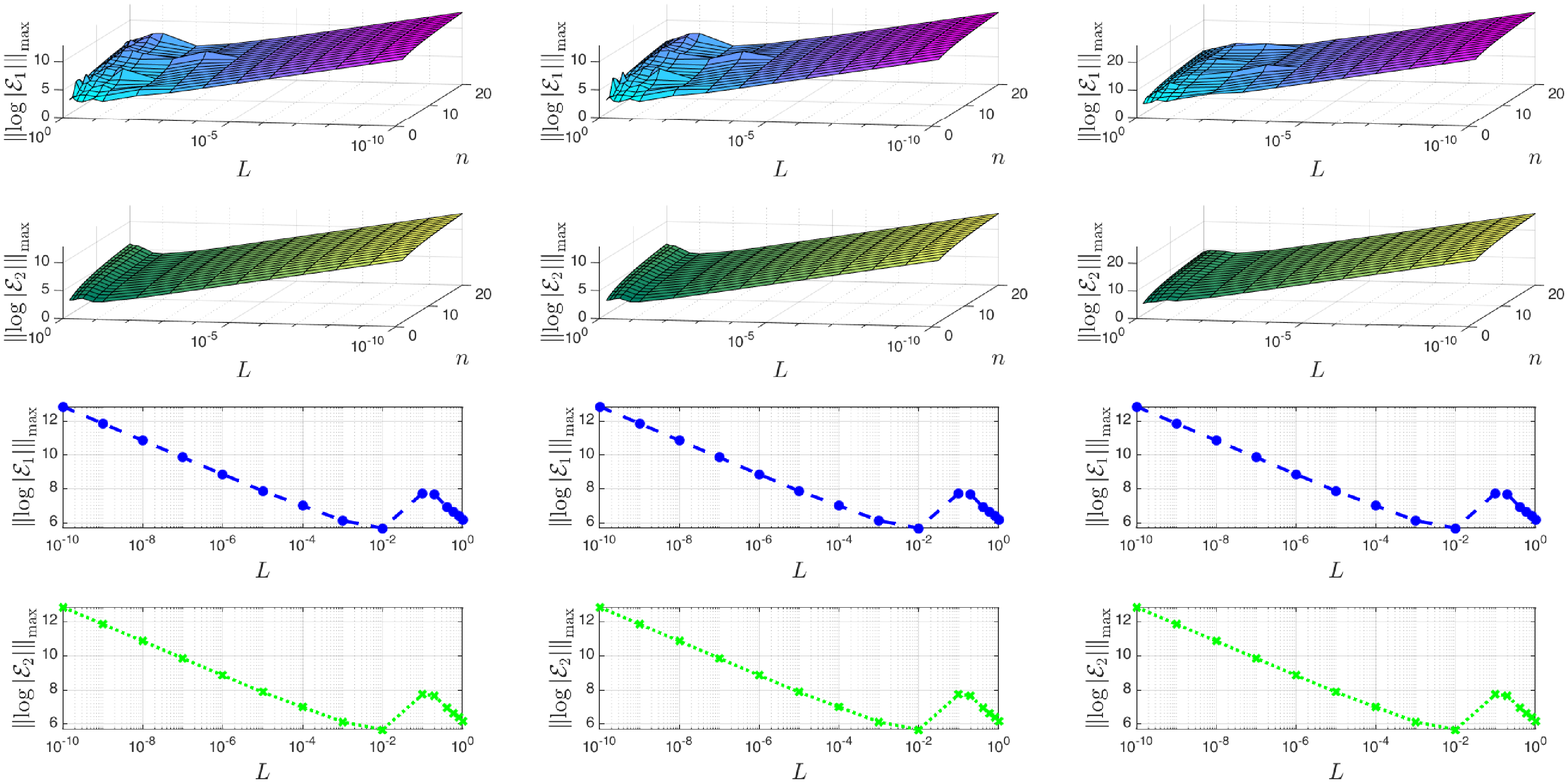}
\caption{First row: The maximum logarithmic errors in evaluating the integrals $I_1$ (left), $I_2$ (middle), and $I_3$ (right) using the RG quadrature with $\alpha = 0$ against $N$ and $L$ for $N = 2$:$20$ and $L = 10^{-10:1:-1},0.2$:$0.2$:$1$. Second row: The corresponding errors using the EG quadrature. Third and fourth rows: The corresponding cross-sections of the maximum logarithmic errors that are shown in the first two rows with the plane $n = 20$.}
\label{fig:Fig1_0}
\end{figure}

\begin{figure}[H]
\centering
\includegraphics[scale=0.4]{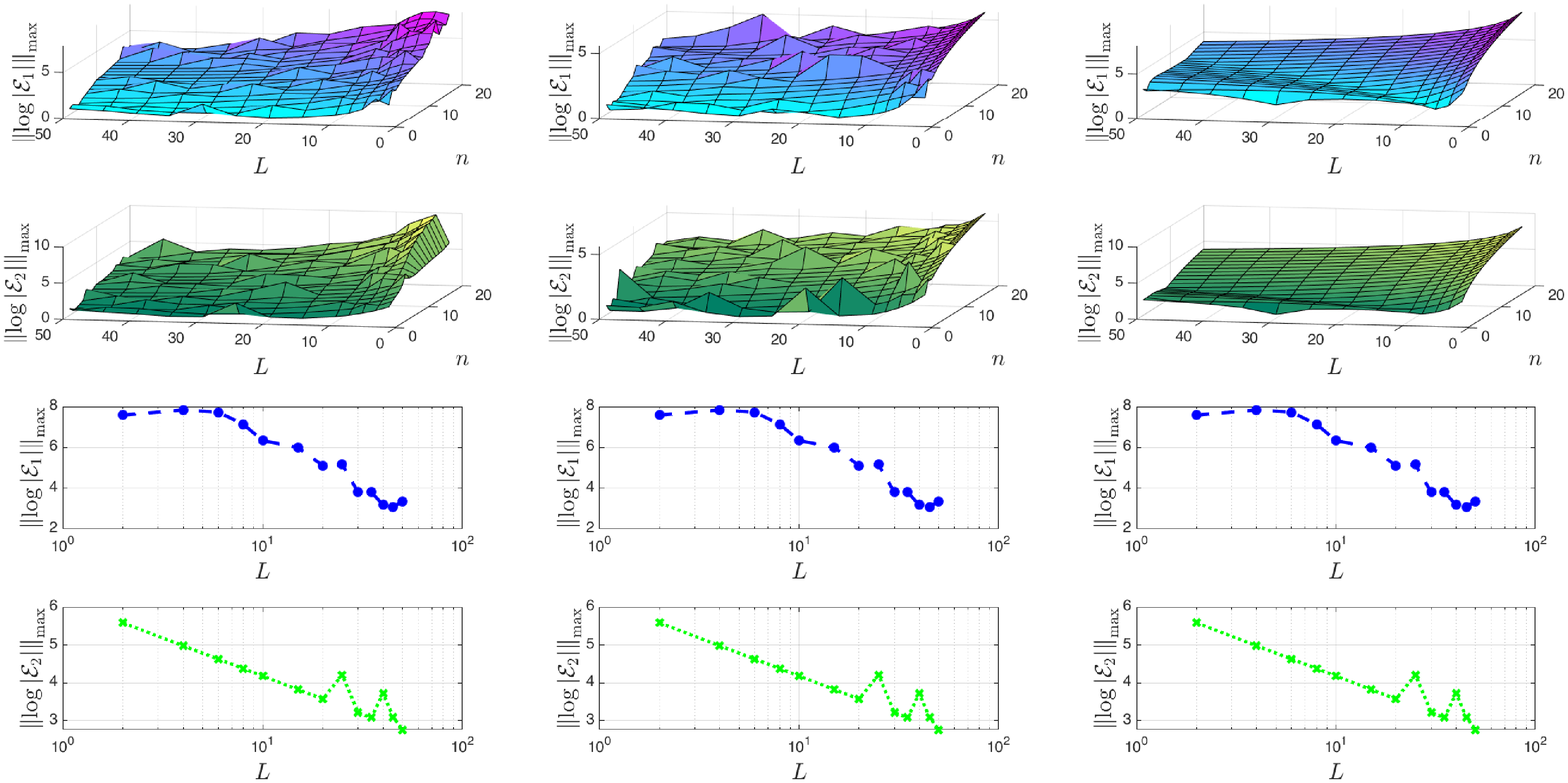}
\caption{First row: The maximum logarithmic errors in evaluating the integrals $I_1$ (left), $I_2$ (middle), and $I_3$ (right) using the RG quadrature with $\alpha = 0$ against $N$ and $L$ for $N = 2$:$20$ and $L = 2$:$2$:$10, 15$:$5$:$50$. Second row: The corresponding errors using the EG quadrature. Third and fourth rows: The corresponding cross-sections of the maximum logarithmic errors that are shown in the first two rows with the plane $n = 20$.}
\label{fig:Fig2_0}
\end{figure}

\begin{figure}[H]
\centering
\includegraphics[scale=0.4]{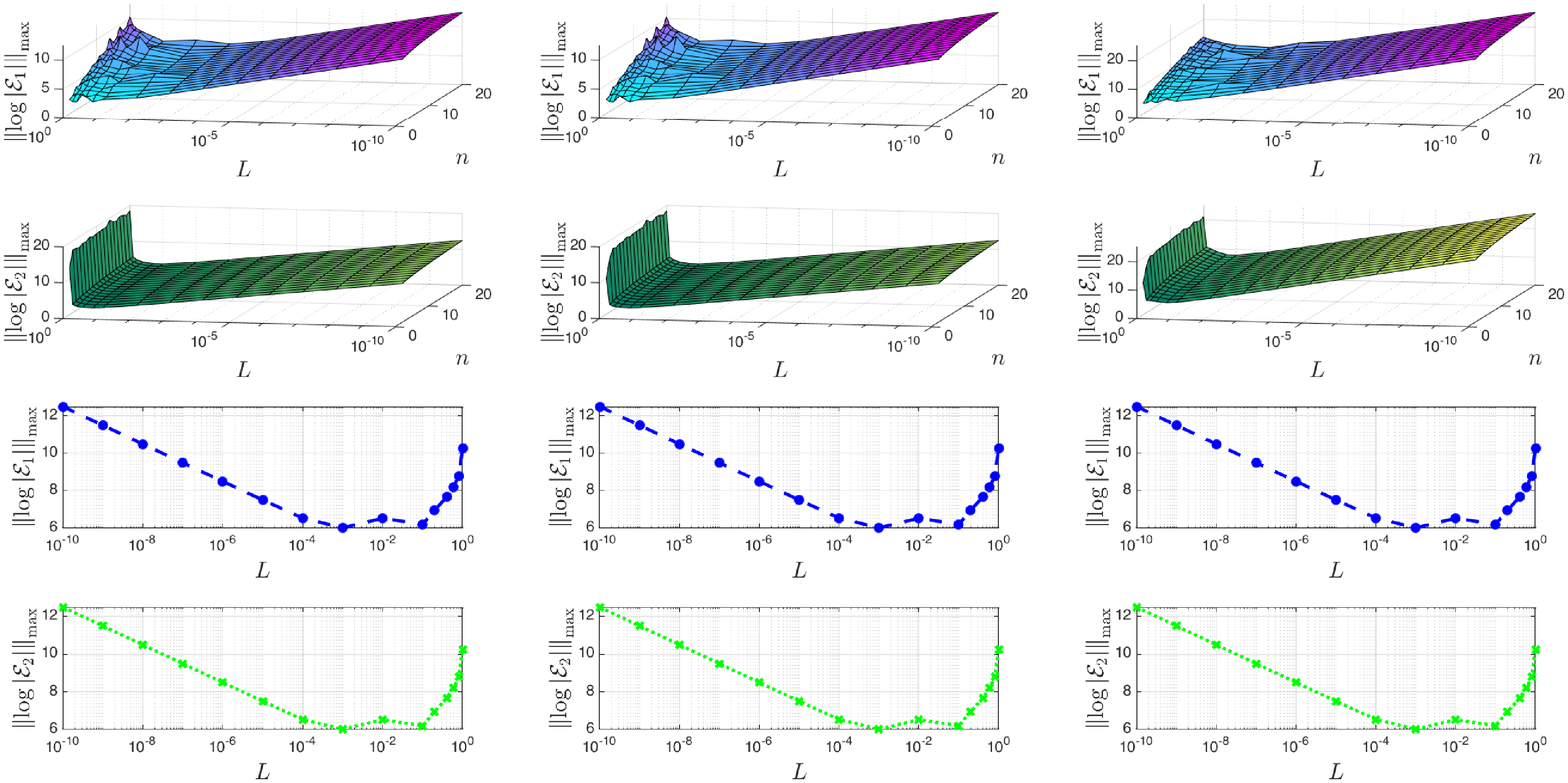}
\caption{First row: The maximum logarithmic errors in evaluating the integrals $I_1$ (left), $I_2$ (middle), and $I_3$ (right) using the RG quadrature with $\alpha = 0.5$ against $N$ and $L$ for $N = 2$:$20$ and $L = 10^{-10:1:-1},0.2$:$0.2$:$1$. Second row: The corresponding errors using the EG quadrature. Third and fourth rows: The corresponding cross-sections of the maximum logarithmic errors that are shown in the first two rows with the plane $n = 20$.}
\label{fig:Fig1_0p5}
\end{figure}

\begin{figure}[H]
\centering
\includegraphics[scale=0.4]{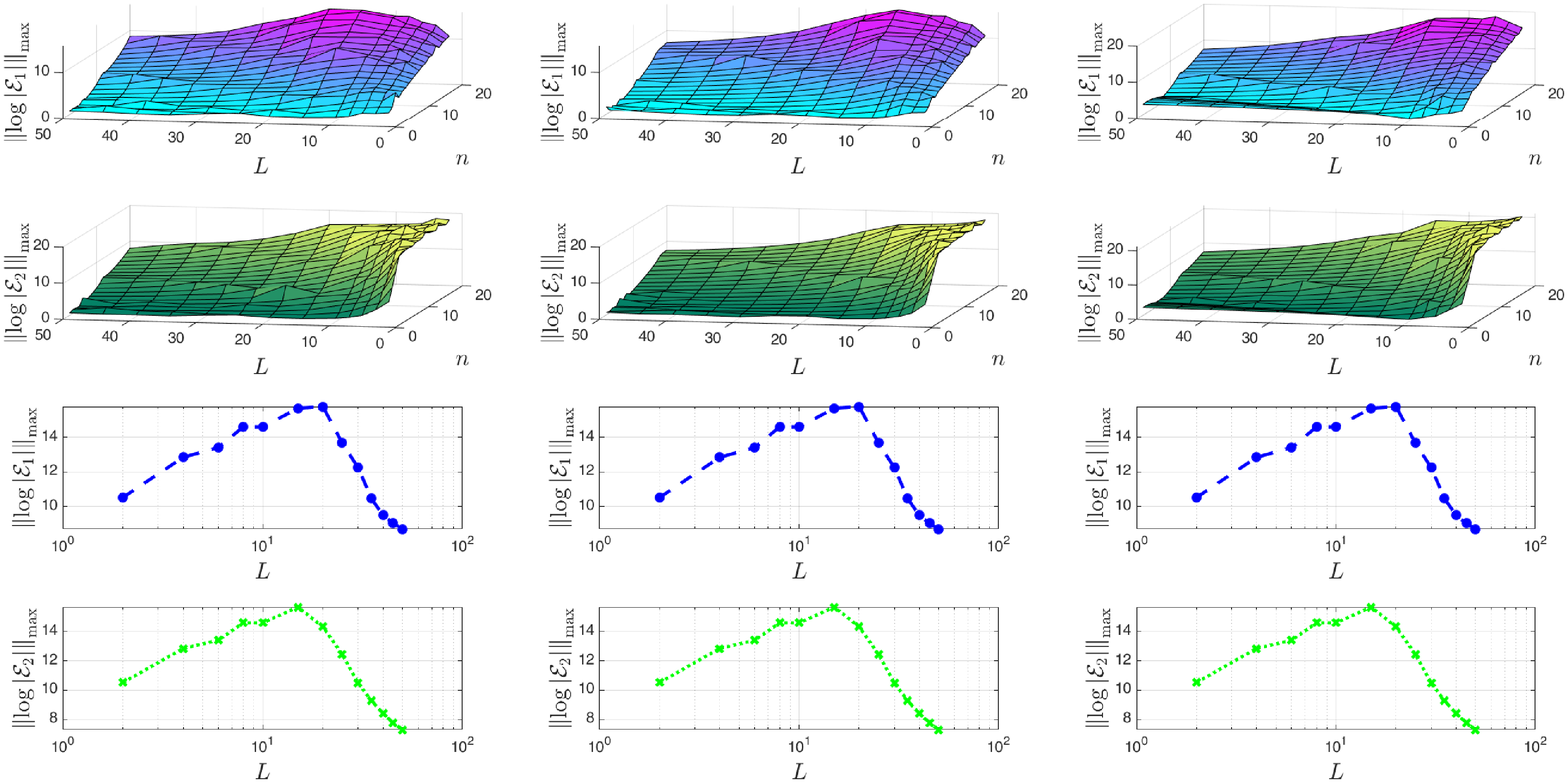}
\caption{First row: The maximum logarithmic errors in evaluating the integrals $I_1$ (left), $I_2$ (middle), and $I_3$ (right) using the RG quadrature with $\alpha = 0.5$ against $N$ and $L$ for $N = 2$:$20$ and $L = 2$:$2$:$10, 15$:$5$:$50$. Second row: The corresponding errors using the EG quadrature. Third and fourth rows: The corresponding cross-sections of the maximum logarithmic errors that are shown in the first two rows with the plane $n = 20$.}
\label{fig:Fig2_0p5}
\end{figure}

\begin{figure}[H]
\centering
\includegraphics[scale=0.4]{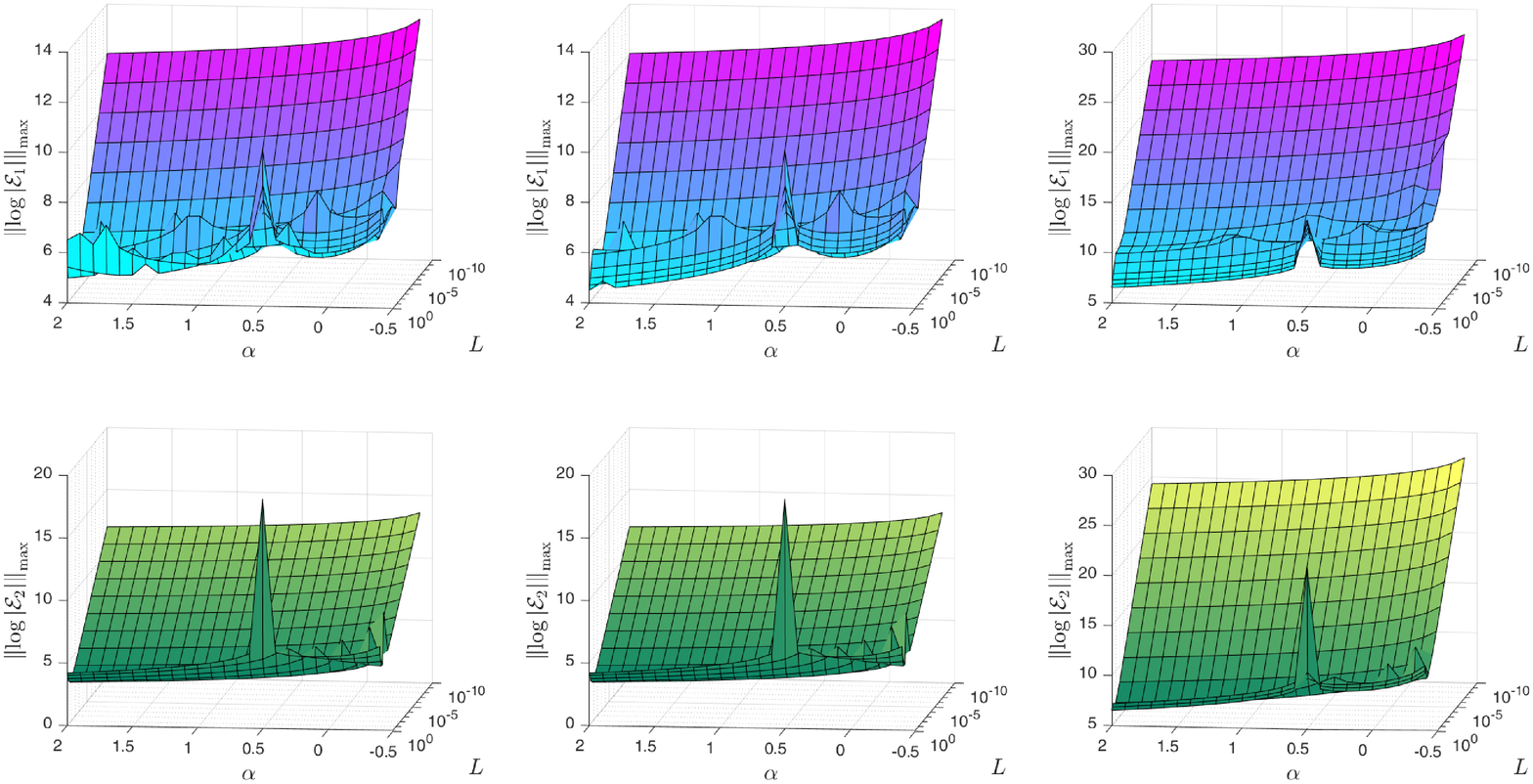}
\caption{First row: The maximum logarithmic errors in evaluating the integrals $I_1$ (left), $I_2$ (middle), and $I_3$ (right) using the RG quadrature with $n = 20$ against $\alpha$ and $L$ for $\alpha = -0.4$:$0.1$:$2$ and $L = 10^{-10:1:-1},0.2$:$0.2$:$1$. Second row: The corresponding errors using the EG quadrature.}
\label{fig:Fig3n20LS}
\end{figure}

\begin{figure}[H]
\centering
\includegraphics[scale=0.4]{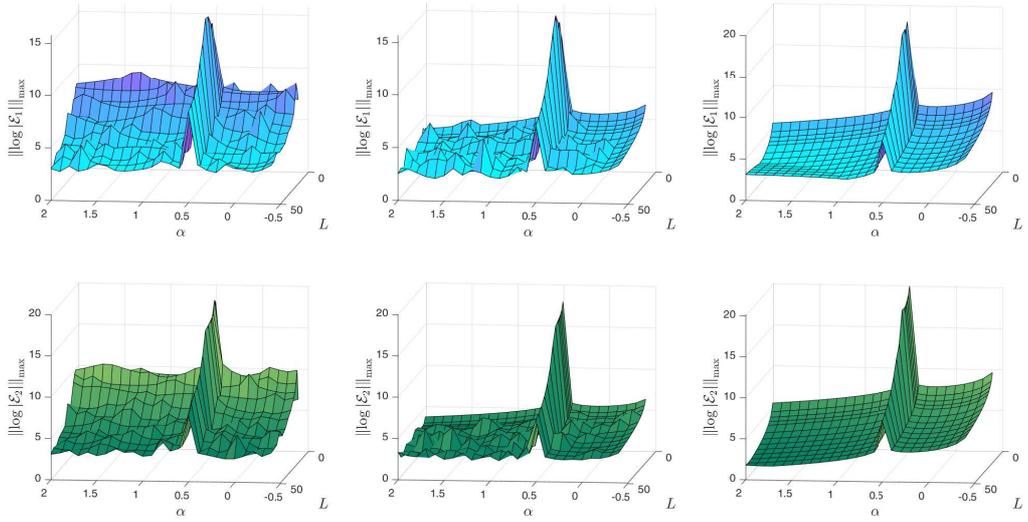}
\caption{First row: The maximum logarithmic errors in evaluating the integrals $I_1$ (left), $I_2$ (middle), and $I_3$ (right) using the RG quadrature with $n = 20$ against $\alpha$ and $L$ for $\alpha = -0.4$:$0.1$:$2$ and $L = 2$:$2$:$10, 15$:$5$:$50$. Second row: The corresponding errors using the EG quadrature.}
\label{fig:Fig3n20LL}
\end{figure}

\subsection{Benchmark Test Problems}
\label{subsec:BP1}
In this section, we show the results of the EG-IS method for solving two benchmark test problems that were studied in \cite{cha2019infinite}. Similar results of the EG-IS method can be obtained by the RG-IS method by following Rule of Thumb \ref{rot:1}.\medskip

\textbf{Problem 1.} Consider Problem $\SCR{A}$ with $\F{Q} = \F{I}_5, \F{R} = \F{I}_4, \F{C} = [1,0,0,0,0],\,\F{D}^{\top} = [0,1,1,1,1]$,
\[\F{A} = \left[ {\begin{array}{*{20}{c}}
{ - 964.8}&{ - 33985.7}&{ - 33985.7}&{ - 33985.7}&{ - 33985.7}\\
0&{ - 400}&0&0&0\\
0&0&{ - 400}&0&0\\
0&0&0&{ - 400}&0\\
0&0&0&0&{ - 400}
\end{array}} \right],\quad \text{and }\F{B} = \left[ {\begin{array}{*{20}{c}}
0&0&0&0\\
{400}&0&0&0\\
0&{400}&0&0\\
0&0&{400}&0\\
0&0&0&{400}
\end{array}} \right].\]
This problem models a DCS linearized at a nominal point. Here $\bmx(t) = [p(t), A^{\top}_{1:4}\cancbra{t}]^{\top}, \bmu(t) = u_{1:4}\cancbra{t}$, and $\bmy(t) = p(t)$, where $p(t)$ represents the change of the pressure within the combustion chamber of the gas generator, $A_i(t)$ denotes the $i$th nozzle's area change from the nominal value, and $u_i(t)$ is the command to $i$th servo $\foralle i \in \MBN_4$; cf. \cite{cha2019infinite}. 
Figure \ref{fig:BTPFig1} shows a schematic diagram of the DCS. Note that the design-state constraint \eqref{eq:3} for this problem indicates that the sum of the changes in the throat area of the four nozzles must be zero while regulating the DCS. \citet{cha2019infinite} solved the problem by solving the existence conditions for constraining output feedback gains and obtained the approximate optimal gain ${\F{K}^*}^{\top} = [2.5125,-1.4994,-2.4537,1.4406]$ after a few trials. The corresponding optimal performance index was estimated to be $232.2286$, rounded to four decimal digits using the initial condition $\bmx(0) = [200, 10,-10,5,-5]^{\top}$. Following the Rule of Thumb \ref{rot:1}, we initially ran the TG-IS method with negative $\alpha$ values at low mesh grids for values of $0 < L < 1$ and then allowed $\alpha$ to gradually increase as the mesh size grows larger until it reaches zero value to maintain high accuracy and stability. Figure \ref{fig:BTP1} shows plots of the approximate optimal state and control variables obtained by the proposed EG-IS method using $L = 0.025$ and $(n,\alpha) \in \{(20,-0.45), (40,-0.1), (80,0)\}$ under the same initial condition. The figure clearly shows the exponential decay of all state variables in a short time, while the optimal control inputs initially increase rapidly for a short time period before decaying exponentially as the time increases. Figure \ref{fig:BTP2} shows the plots of the corresponding size of the four nozzle throat area change sum, which nearly reaches the round-off error plateau, reflecting the satisfaction of the design state constraint \eqref{eq:3} with excellent accuracy. 
Figure \ref{fig:BTP4} shows the corresponding plots of the MAEs when approximating the discrete weak dynamical system equations of NLP \eqref{prob:Opt1} at the collocation points. The corresponding approximate optimal output feedback gains obtained by the EG-IS method are listed in Table \ref{tab:Table1}, where the recorded data converge to ${\bmK^*}^{\top} \approx [{0.01},-0.01,{0.005},{-0.005}]$, rounded to three decimal digits when $(n,\alpha)$ reaches $(80,0)$. The corresponding recorded elapsed times were about $13.65$ s, $3.04$ min, and $12.92$ min, respectively. Figure \ref{fig:BTP5} shows the profile of the approximate optimal performance index for increasing $n$ values and gradually increasing nonpositive $\alpha$ values, where we observe that $J_{120} \approx 69.3811$, rounded to four decimal digits--  a reduction of nearly $70$\% of the estimated optimal cost function value obtained earlier in \cite{cha2019infinite}.

\begin{figure}[H]
\centering
\includegraphics[scale=0.45]{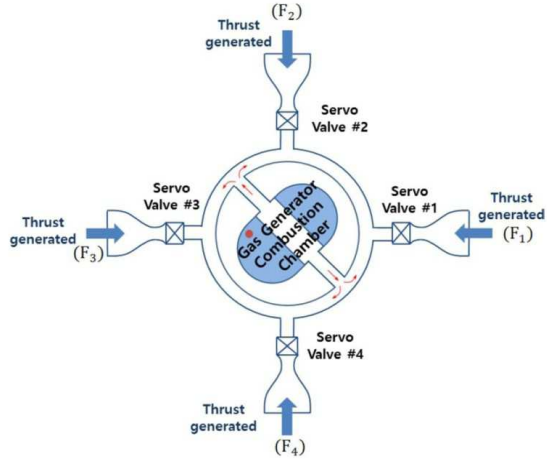}
\caption{Schematic figure of the DCS as quoted from \cite{cha2019infinite}.}
\label{fig:BTPFig1}
\end{figure}

\begin{figure}[H]
\centering
\includegraphics[scale=0.4]{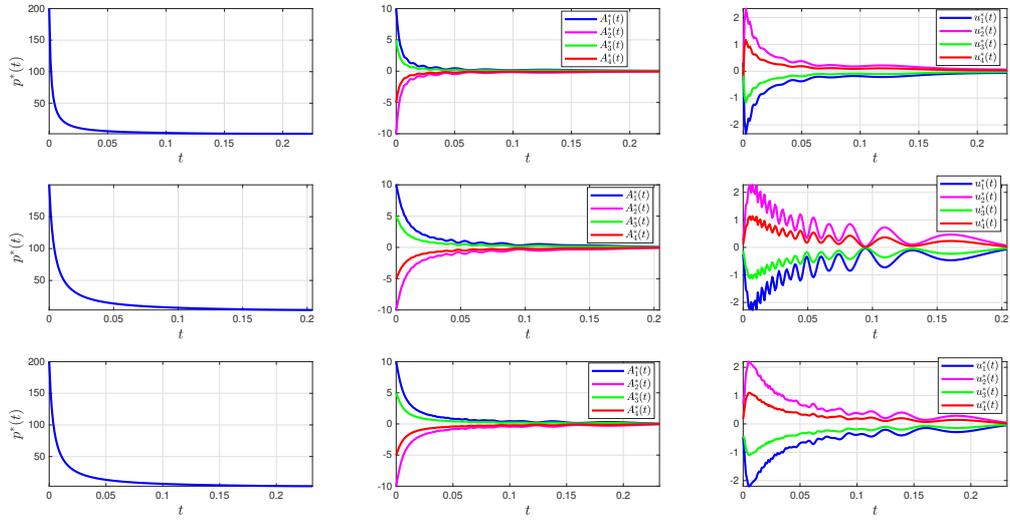}
\caption{Plots of the approximate optimal state and control variables obtained by the EG-IS method using $L = 0.025$ and $(n,\alpha) = (20,-0.45)$ (first row), $(n,\alpha) = (40,-0.1)$ (second row), and $(n,\alpha) = (80,0)$ (third row). The plots were generated using $100$ equally-spaced nodes from $0$ to ${}_2t_{n,n}^{\alpha}$.}
\label{fig:BTP1}
\end{figure}

\begin{figure}[H]
\centering
\includegraphics[scale=0.5]{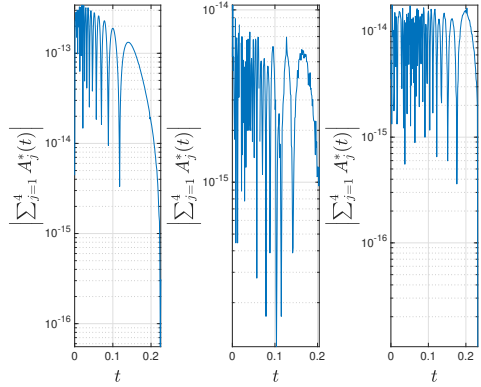}
\caption{Plots of the size of the four nozzle's throat area changes sum obtained by the EG-IS method using $L = 0.025$ and $(n,\alpha) = (20,-0.45)$ (left), $(n,\alpha) = (40,-0.1)$ (middle), and $(n,\alpha) = (80,0)$ (right). The plots were generated using $100$ equally-spaced nodes from $0$ to ${}_2t_{n,n}^{\alpha}$.}
\label{fig:BTP2}
\end{figure}


\begin{figure}[H]
\centering
\includegraphics[scale=0.5]{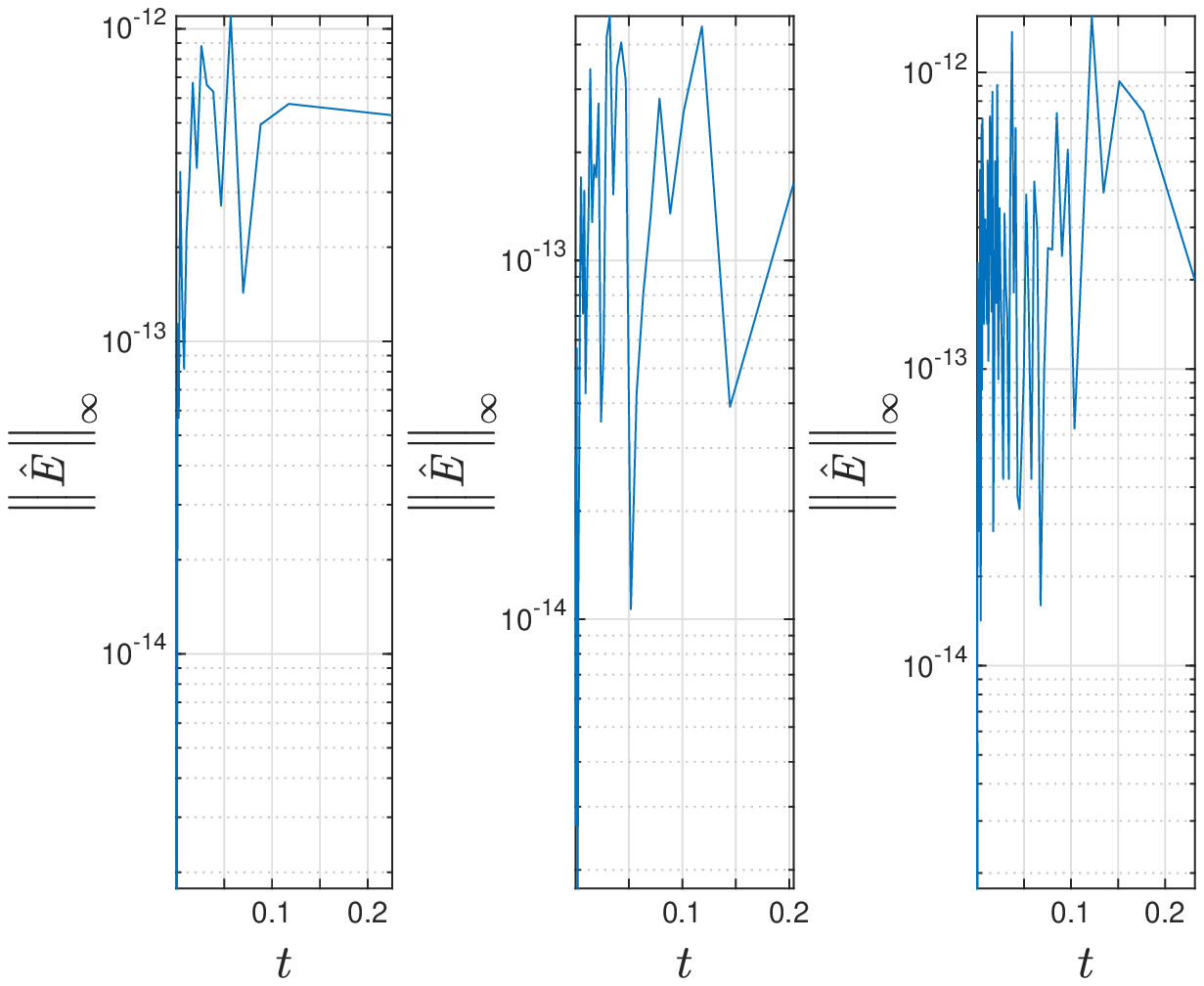}
\caption{Plots of the MAEs produced by the EG-IS method in approximating the discrete weak dynamical system equations of NLP \eqref{prob:Opt1} at the collocation points using $L = 0.025$ and $(n,\alpha) = (20,-0.45)$ (left), $(n,\alpha) = (40,-0.1)$ (middle), and $(n,\alpha) = (80,0)$ (right). The plots were generated using $100$ equally-spaced nodes from $0$ to ${}_2t_{n,n}^{\alpha}$.}
\label{fig:BTP4}
\end{figure}

\begin{table}[H]
\caption{The approximate optimal output feedback gains obtained by the EG-IS method using $L = 0.025$ and several parameter values. All approximations are rounded to three decimal digits.}
\centering
\resizebox{0.5\textwidth}{!}{  
\begin{tabular}{ccc}  
\toprule
$(n,\alpha) = (20,-0.45)$ & $(n,\alpha) = (40,-0.1)$ & $(n,\alpha) = (80,0)$ \\
$\bmK^*$ & $\bmK^*$ & $\bmK^*$ \\
\midrule
$\left[ {\begin{array}{*{20}{c}}
{0.008}\\
{ -0.008}\\
{0.004}\\
{ -0.004}
\end{array}} \right]$ & $\left[ {\begin{array}{*{20}{c}}
{0.009}\\
{ -0.009}\\
{0.005}\\
{ -0.005}
\end{array}} \right]$ & $\left[ {\begin{array}{*{20}{c}}
{0.010}\\
{ -0.010}\\
{0.005}\\
{ -0.005}
\end{array}} \right]$ \\
\bottomrule
\end{tabular}
}
\label{tab:Table1}
\end{table}

\begin{figure}[H]
\centering
\includegraphics[scale=0.5]{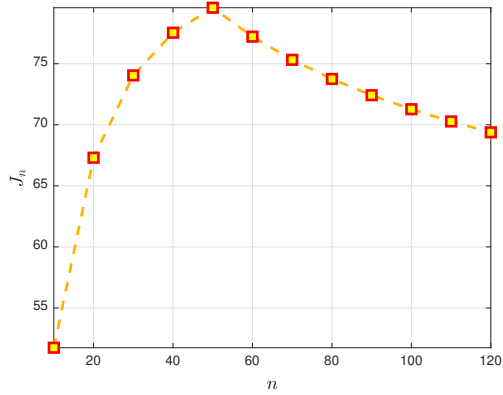}
\caption{The approximate optimal performance index values against $n$ for $L = 0.025$ and $(n,\alpha) \in \{(10,-0.4), (20,-0.3), (30,-0.2), (40,-0.1), (50:120,0)\}$.}
\label{fig:BTP5}
\end{figure}

\textbf{Problem 2.} Consider Problem $\SCR{A}$ with $\F{Q} = \diag{50, 100, 100, 50, 0, 0, 1}, \F{R} = 0.1 \F{I}_2, \F{D}^{\top} = [0,0,0,0,12,-1,0]$, 
\begin{gather*}
\F{A} = \left[ {\begin{array}{*{20}{c}}
{ - 0.3220}&{0.0640}&{0.0364}&{ - 0.9917}&{0.0003}&{0.0008}&0\\
0&0&1&{0.0037}&0&0&0\\
{ - 30.6492}&0&{ - 3.6784}&{0.6646}&{ - 0.7333}&{0.1315}&0\\
{8.5396}&0&{ - 0.0254}&{ - 0.4764}&{ - 0.0319}&{{\rm{ }} - 0.0620}&0\\
0&0&0&0&{ - 20.2}&0&0\\
0&0&0&0&0&{ - 20.2}&0\\
0&0&0&{57.2958}&0&0&{ - 1}
\end{array}} \right],\quad \F{B} = \left[ {\begin{array}{*{20}{c}}
0&0\\
0&0\\
0&0\\
0&0\\
{20.2}&0\\
0&{20.2}\\
0&0
\end{array}} \right],\\
\text{and }\F{C} = \left[ {\begin{array}{*{20}{c}}
0&0&0&{57.2958}&0&0&{ - 1}\\
0&0&{57.2958}&0&0&0&0\\
{57.2958}&0&0&0&0&0&0\\
0&{57.2958}&0&0&0&0&0
\end{array}} \right].
\end{gather*}
This problem models the lateral dynamics of an F-16 aircraft linearized 
about the nominal flight condition (502 ft$/$s and 300 psf of dynamic pressure) retaining the lateral states sideslip $\beta$, bank angle $\phi$, roll rate $p$, and yaw rate $r$. Here $\bmx = [\beta, \phi, p, r, \delta_a, \delta_r, x_w]^{\top}, \bmu = [u_a, u_r]^{\top}$, and $\bmy = [r_w, p, \beta, \phi]^{\top}$, where $\delta_a$ and $\delta_r$ are additional states introduced by the aileron and rudder actuators such that $(\delta_a,\delta_r) = \mu (u_a, u_r)$ with $\mu(t) = 20.2/(t+20.2)$ and $r_w = t/(t+1) r$ is the washed-out yaw rate; cf. \cite{cha2019infinite}. The design state constraint \eqref{eq:3} for this problem indicates that the aileron and the rudder are linked to each other by the functional relationship 
\begin{equation}\label{eq:ARI1new1}
\delta_r = 12 \delta_a,
\end{equation}
while regulating the system, which is known as an aileron-rudder interconnect (ARI) often used for stabilizing the system at a high angle of attack. Through the same approach, \citet{cha2019infinite} obtained the approximate optimal gain 
\[
{\F{K}^*} = \left[ {\begin{array}{*{20}{c}}
{ - 0.2944}&{ - 0.0592}&{0.0357}&{ - 0.0271}\\
{ - 3.5333}&{ - 0.7109}&{0.4284}&{ - 0.3247}
\end{array}} \right],
\]
after a few trials. The corresponding optimal performance index was estimated to be $428.7008$, rounded to four decimal digits, using the initial condition $\bmx(0) = [0.5, \bmzer_6^{\top}]^{\top}$. Following Rule of Thumb \ref{rot:1}, we ran the EG-IS method using the recommended parameter values $(\alpha,L) = (1/2,15)$ for increasing mesh grids to achieve possibly the highest accuracy and stability based on the numerical observations prescribed in Section \ref{subsec:PPTQSF1} and supported by the stability analysis. Figures \ref{fig:Fig11} and \ref{fig:Fig12} show the plots of the approximate optimal state and control variables obtained by the proposed EG-IS method using $(\alpha,L) = (1/2,15), n \in \{20, 40, 80\}$, and the same initial condition. Figure \ref{fig:Fig13} shows the plots of the size of the errors by which the approximate optimal solutions fail to satisfy the ARI Condition \eqref{eq:ARI1new1}. Figure \ref{fig:Fig14} shows the corresponding plots of the MAEs in approximating the discrete weak dynamical system equations of NLP \eqref{prob:Opt1} at the collocation points. The corresponding approximate optimal output feedback gains obtained by the EG-IS method are shown in Table \ref{tab:Table2}, where the recorded data converge to 
\[{\bmK^*} = \left[ {\begin{array}{*{20}{c}}
{0.026}&{0.054}&{0.136}&{ - 0.001}\\
{0.312}&{0.653}&{1.637}&{ - 0.015}
\end{array}} \right],\]
rounded to three decimal digits, when $n$ reaches $80$. The corresponding recorded elapsed times were about $8.72$ s, $1.59$ min, and $21.32$ min, respectively. Figure \ref{fig:Fig15} shows the profile of the approximate optimal performance index for increasing $n$ values, where we notice the exponential convergence of $J_n$, converging to nearly $316.8154$, rounded to four decimal digits, as early as $n$ reaches $20$--  a reduction of nearly $26$\% of the estimated optimal cost function value obtained earlier in \cite{cha2019infinite}.

\begin{rem}
It is noteworthy to mention that for systems of small or moderate mesh size, the solutions by the RG-IS and EG-IS methods are easy and fast to implement; however, for large grids, their reduced NLPs become large-scale and MATLAB fmincon solver runs relatively slow. Fortunately, the EG-IS method converges to near-optimal results for sufficiently smooth solutions using a relatively small number of grids as shown by Tables \ref{tab:Table1} and \ref{tab:Table2}, which allow us to achieve satisfactorily accurate approximations in relatively short running times. We expect a similar behavior for the RG-IS method by following Rule of Thumb \ref{rot:1}.
\end{rem}

\begin{figure}[H]
\centering
\includegraphics[scale=0.4]{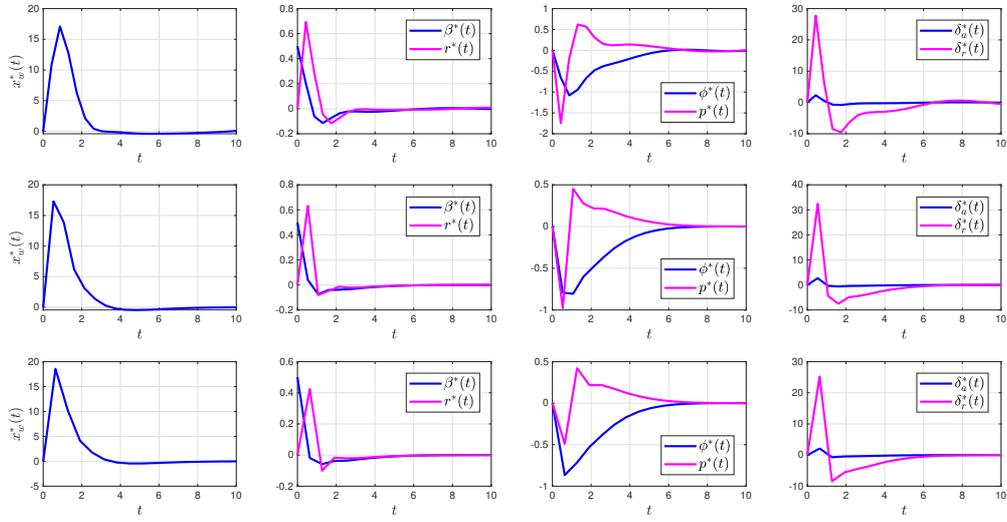}
\caption{Plots of the approximate optimal state variables obtained by the EG-IS method using $(\alpha,L) = (1/2,15)$ and $n = 20$ (first row), $n = 40$ (second row), and $n = 80$ (third row). The plots were generated using $100$ equally-spaced nodes from $0$ to ${}_2t_{n,n}^{0.5}$.}
\label{fig:Fig11}
\end{figure}

\begin{figure}[H]
\centering
\includegraphics[scale=0.4]{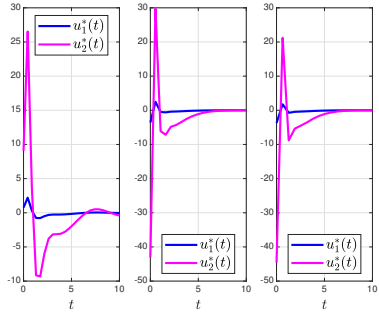}
\caption{Plots of the approximate optimal control variables obtained by the EG-IS method using $(\alpha,L) = (1/2,15)$ and $n = 20$ (left), $n = 40$ (middle), and $n = 80$ (right). The plots were generated using $100$ equally-spaced nodes from $0$ to ${}_2t_{n,n}^{0.5}$.}
\label{fig:Fig12}
\end{figure}

\begin{figure}[H]
\centering
\includegraphics[scale=0.5]{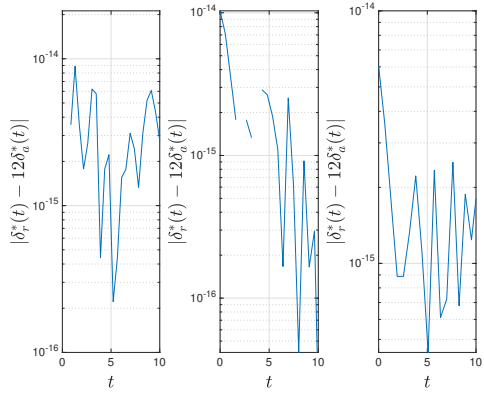}
\caption{Plots of the size of the errors by which the approximate optimal solutions fail to satisfy the ARI Condition \eqref{eq:ARI1new1}. The errors were produced by the EG-IS method using $(\alpha,L) = (1/2,15)$ and $n = 20$ (left), $n = 40$ (middle), and $n = 80$ (right). The plots were generated using $100$ equally-spaced nodes from $0$ to ${}_2t_{n,n}^{0.5}$.}
\label{fig:Fig13}
\end{figure}

\begin{figure}[H]
\centering
\includegraphics[scale=0.5]{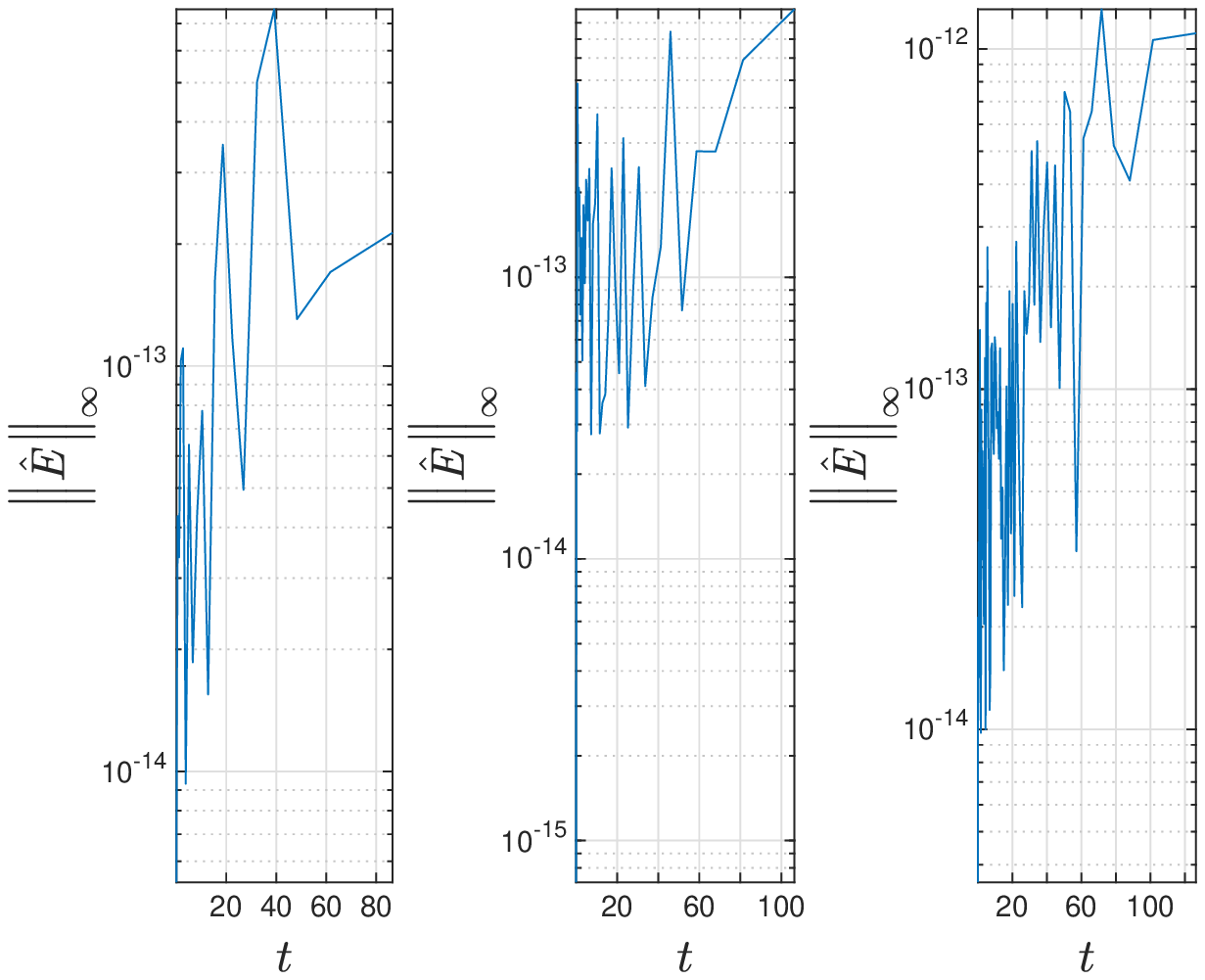}
\caption{Plots of the MAEs produced by the EG-IS method in approximating the discrete weak dynamical system equations of NLP \eqref{prob:Opt1} at the collocation points using $(\alpha,L) = (1/2,15)$ and $n = 20$ (left), $n = 40$ (middle), and $n = 80$ (right). The plots were generated using $100$ equally-spaced nodes from $0$ to ${}_2t_{n,n}^{0.5}$.}
\label{fig:Fig14}
\end{figure}

\begin{table}[H]
\caption{The approximate optimal output feedback gains obtained by the EG-IS method using $(\alpha,L) = (1/2,15)$ and several $n$ values. All approximations are rounded to three decimal digits.}
\centering
\resizebox{0.5\textwidth}{!}{  
\begin{tabular}{ccc}  
\toprule
$n = 20$ & $n = 40$ & $n = 80$ \\
${\bmK^*}^{\top}$ & ${\bmK^*}^{\top}$ & ${\bmK^*}^{\top}$ \\
\midrule
$\left[ {\begin{array}{*{20}{c}}
{0.037}&{0.441}\\
{0.031}&{0.372}\\
{ - 0.017}&{ - 0.201}\\
{ - 0.001}&{ - 0.017}
\end{array}} \right]$ & $\left[ {\begin{array}{*{20}{c}}
{0.026}&{0.312}\\
{0.055}&{0.654}\\
{0.137}&{1.640}\\
{ - 0.001}&{ - 0.015}
\end{array}} \right]$ & $\left[ {\begin{array}{*{20}{c}}
{0.026}&{0.312}\\
{0.054}&{0.653}\\
{0.136}&{1.637}\\
{ - 0.001}&{ - 0.015}
\end{array}} \right]$\\
\bottomrule
\end{tabular}
}
\label{tab:Table2}
\end{table}

\begin{figure}[H]
\centering
\includegraphics[scale=0.5]{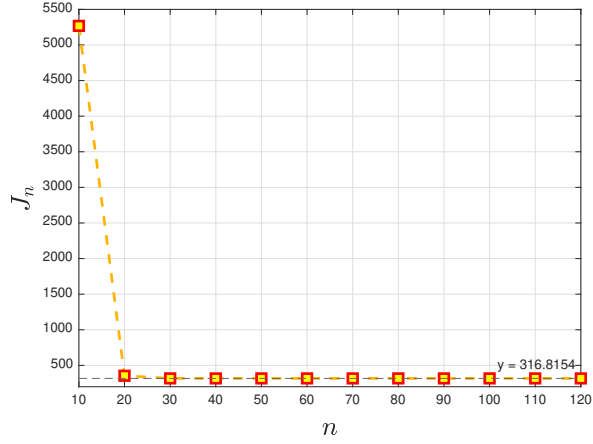}
\caption{The approximate optimal performance index values against $n$ for $(\alpha,L) = (1/2,15)$ and $n = 10$:$10$:$120$.}
\label{fig:Fig15}
\end{figure}

\section{Conclusion}
\label{sec:Conc}
This study introduces two novel numerical methods for solving a class of infinite-horizon regulation problems governed by linear systems with state equality constraints and output feedback control. The proposed TG-IS method proved to be successful in practice and was capable of providing accurate numerical solutions with rapid convergence. In particular, the TG-IS method is more efficient than the TG-IPS method, which leads to NLPs with dimensions proportional to the square of the collocation mesh size, whereas the former method produces NLPs with dimensions that grow only linearly with the collocation mesh size. Another merit of the TG-IS method occurs in its ability to provide a useful means to avoid the derivation and solution of the sophisticated optimal gain design equations of \citet{cha2019infinite} for the constrainable output feedback gains, which require either finding the optimal form of a nonunique basis matrix or fixing a basis matrix a priori before finding the optimal gain and then compare the effect of the various basis matrices on the performance of the controller via numerical simulations. Instead, the TG-IS method directly discretizes the integral form of the IHOC in its original domain via efficient collocation and highly accurate quadratures based on TG functions. Based on the rigorous stability analysis presented in \ref{sec:STGI1}, we proved that the TG interpolation of a function is asymptotically stable for $\alpha \in \MBRzerP$ under mild conditions. However, such schemes are asymptotically unstable for $(n,\alpha) \in\,\MBZPL \times \MBRmh$, especially when $\alpha \to -1/2$. These theoretical results, in addition to the extensive numerical simulations in Section \ref{subsec:PPTQSF1}, led to the foundation of Rule of Thumb \ref{rot:1}, which gave us a practical way to optimize the performance of the TG-IS method by identifying the optimal ranges of the mesh size, the Gegenbauer index, and the mapping scaling parameter associated with the TG functions.

\section{Future Work}
\label{sec:FW1}
One possible future research direction is to extend the application of the proposed methods to handle more general nonlinear IHOCs.

\appendix
\section{Rational and Exponential Gegenbauer Functions}
\label{sec:TOTIHOCP1}
We define the RG and EG functions with ordered index pair $(\alpha,L) \in\,\MBRmh \times \MBRP$ by 
\begin{equation}\label{eq:5}
\C{G}_{i,n}^{\alpha,L}(t) = G_{n}^{\alpha}\left(T_{i,L}(t)\right)\quad \forall (i,n,t) \in \MBN_2 \times \MBZzerP \times \MBRzerP,
\end{equation}
respectively, where
\begin{equation}\label{eq:trans}
{T_{i,L}}(t) = \left\{ \begin{array}{l}
(t - L)/(t + L),\quad i = 1,\\
1 - 2{e^{ - t/L}},\quad i = 2,
\end{array} \right.
\end{equation}
and $L$ is a mapping scaling parameter. Both functions are referred to collectively by ``the TG functions,'' or more specifically by ``the TG functions with ordered index pair $(\alpha,L)$.'' The transformed Chebyshev (TC) function and transformed Legendre (TL) function are special cases of the TG function for the Gegenbauer parameter values $\alpha = 0$ and $\alpha = 1/2$, respectively. The TG function $\C{G}_{i,n}^{\alpha,L}$ appears as an eigensolution to the following singular Sturm-Liouville (SSL) problem on $\MBRzerP$:
\begin{equation}\label{eq:SE1}
\frac{1}{{T'_{i,L}(t)}}\frac{d}{{dt}}\left[ {\frac{{{{\left( {1 - T_{i,L}^2(t)} \right)}^{\alpha  + 1/2}}}}{{{T'_{i,L}}(t)}} {\C{G}'_{i,n}}^{\hspace{-1mm}\alpha,L}(t)} \right] + n\left( {n + 2\alpha } \right){\left( {1 - T_{i,L}^2(t)} \right)^{\alpha  - 1/2}}\C{G}_{i,n}^{\alpha,L}(t) = 0\quad \forall i \in \MBN_2,
\end{equation}
where 
\begin{equation}\label{eq:trans2}
{T'_{i,L}}(t) = \left\{ \begin{array}{l}
2 L/(t + L)^2,\quad i = 1,\\
2 {e^{ - t/L}}/L,\quad i = 2.
\end{array} \right.
\end{equation}
In particular, the RG and EG functions are eigenfunctions of the following SSL problems
\begin{subequations}
\begin{gather}
\frac{1}{L}{(t + L)^2}\frac{d}{{dt}}\left[{t\,{{\left( {\frac{t}{{{{\left( {t + L} \right)}^2}}}} \right)}^{\alpha  - \frac{1}{2}}}{\C{G}'_{1,n}}^{\hspace{-1.8mm}\alpha,L}(t)} \right] + n\left( {n + 2\alpha } \right){\left( {\frac{t}{{{{\left( {t + L} \right)}^2}}}} \right)^{\alpha  - \frac{1}{2}}}\C{G}_{1,n}^{\alpha,L}(t) = 0,\\
\frac{{{L^2}}}{4}{e^{t/L}}\frac{d}{{dt}}\left[ {{e^{\frac{t}{L}}}\,{{\left( {1 - {{\left( {1 - 2{e^{ - \frac{t}{L}}}} \right)}^2}} \right)}^{\alpha  + \frac{1}{2}}}{\C{G}'_{2,n}}^{\hspace{-1.8mm}\alpha,L}(t)} \right] + n\left( {n + 2\alpha } \right){\left( {1 - {{\left( {1 - 2{e^{ - \frac{t}{L}}}} \right)}^2}} \right)^{\alpha  - \frac{1}{2}}}\C{G}_{2,n}^{\alpha,L}(t) = 0,
\end{gather}
\end{subequations}
which can be restated as follows:
\begin{subequations}
\begin{gather}
\left( {t + L} \right)\left[ {2t\left( {t + L} \right){\C{G}''_{1,n}}^{\hspace{-0.8mm}\alpha,L}(t) + \left( {(2\alpha  + 1)L + (3 - 2\alpha )t} \right){\C{G}'_{1,n}}^{\hspace{-1.8mm}\alpha,L}(t)} \right] + 2Ln\left( {n + 2\alpha } \right)\C{G}_{1,n}^{\alpha,L}(t) = 0,\\
L\left[ {2L\left( {{{e}^{t/L}} - 1} \right) {\C{G}''_{2,n}}^{\hspace{-0.8mm}\alpha,L}(t) + \left( {4\alpha  + \left( {1 - 2\alpha } \right){e^{t/L}}} \right) {\C{G}'_{2,n}}^{\hspace{-1.8mm}\alpha,L}(t)} \right] + 2n\left( {n + 2\alpha } \right) \C{G}_{2,n}^{\alpha,L}(t) = 0,
\end{gather}
\end{subequations}
respectively. The weight functions of the TG functions are given in respective order by
\begin{equation}
w^{\alpha,L}_{1}(t)=\frac{4^{\alpha}L^{\alpha+1/2} t^{\alpha-1/2}}{(t+L)^{2\alpha+1}},\quad w^{\alpha,L}_{2}(t)=\frac{2}{L}{e^{ - t/L}}{\left[ {1 - {{\left( {1 - 2{e^{ - t/L}}} \right)}^2}} \right]^{\alpha  - 1/2}}.
\end{equation}
The TG functions are orthogonal with respect to the weight functions $w^{\alpha,L}_1$ and $w^{\alpha,L}_2$, respectively, and their orthogonality relations are defined by
\begin{equation}
\C{I}_{\MBRzerP}^{(t)} {\left(\C{G}_{i,m}^{\alpha,L}\, \C{G}_{i,n}^{\alpha,L} w^{\alpha,L}_{i}\right)}=\lambda^{\alpha}_{n} \delta_{mn}\quad \forall i \in \MBN_2,\quad \{m,n\} \subset \MBZzerP, 
\label{eq:orthog1}
\end{equation}
where
\begin{equation}
\lambda^{\alpha}_{n} = \frac{2^{2\alpha-1} n! \Gamma^{2}{(\alpha+\frac{1}{2})}}{(n+\alpha)\Gamma(n+2 \alpha)}.
\label{eq:orthog2}
\end{equation}
Therefore, the TG functions form an orthogonal basis set for $L_{w_i^{\alpha,L}}^{2}\left(\MBRzerP\right)$ so that the TG expansion of a function $f \in L_{w_i^{\alpha,L}}^2\left(\MBRzerP\right)$ can be written as 
\begin{equation}
f(t) = {{\hat f}_{0:\infty}}\,\C{G}_{i,0:\infty}^{\alpha,L}\cancbra{t}:\quad {\hat f_k} = \frac{1}{{\lambda _k^\alpha }}\C{I}_{\MBRzerP}^{(t)}\left( {f\,\C{G}_{i,k}^{\alpha,L}w_i^{\alpha ,L}} \right)\quad\forall (k,i) \in \MBZzerP \times \MBN_2.
\end{equation}
The TG functions can be generated by the following three-term recurrence relation
\begin{equation}
\C{G}_{i,0}^{\alpha,L}(t) = 1,\quad \C{G}_{i,1}^{\alpha,L}(t) = T_{i,L}(t),\quad (n+2\alpha)\, \C{G}_{i,n+1}^{\alpha,L}(t) = 2 (n+\alpha) T_{i,L}(t)\, \C{G}_{i,n}^{\alpha,L}(t)-n\, \C{G}_{i,n-1}^{\alpha,L}(t)\quad \forall i \in \MBN_2, \quad n\geq 1.
\label{eq:rec}
\end{equation}
In particular, the first four RG and EG functions are given by
\begin{subequations}
\begin{gather}
\C{G}_{1,0}^{\alpha,L}(t) = 1,\quad \C{G}_{1,1}^{\alpha,L}(t) = \frac{t-L}{t+L},\quad \C{G}_{1,2}^{\alpha,L}(t) = \frac{1}{{1 + 2\alpha }}\left[ {2\left( {1 + \alpha } \right){{\left( {\frac{{t - L}}{{t + L}}} \right)}^2} - 1} \right],\quad \C{G}_{1,3}^{\alpha,L}(t) = \frac{1}{{{{\left( {t + L} \right)}^3}}}\left[ {{{\left( {t - L} \right)}^3} - \frac{{12Lt\left( {t - L} \right)}}{{1 + 2\alpha }}} \right],\\
\C{G}_{2,0}^{\alpha,L}(t) = 1,\quad \C{G}_{2,1}^{\alpha,L}(t) = 1-2 e^{-t/L},\quad \C{G}_{2,2}^{\alpha,L}(t) = \frac{1}{{1 + 2\alpha }}\left[ {2\left( {1 + \alpha } \right){{\left( {1 - 2{e^{ - \frac{t}{L}}}} \right)}^2} - 1} \right],\\
\C{G}_{2,3}^{\alpha,L}(t) = \frac{1}{{1 + 2\alpha }}\left[ {{e^{ - \frac{{3t}}{L}}}\left( {{e^{\frac{t}{L}}} - 2} \right)\left( {8\left( {2 + \alpha } \right)\left( {1 - {e^{\frac{t}{L}}}} \right) + \left( {1 + 2\alpha } \right){e^{\frac{{2t}}{L}}}} \right)} \right].
\end{gather}
\end{subequations}
We define $\SCR{T}_{i,n}^{\alpha,L} = \Span\left\{\C{G}_{i,1}^{\alpha,L}, \C{G}_{i,2}^{\alpha,L}, \ldots, \C{G}_{i,n}^{\alpha,L}\right\}$. We also define $\MBTG_{i,n}^{\alpha,L} = \left\{{}_i^Lt_{n,0:n}^{\alpha}: {}_i^Lt_{n,j}^{\alpha} = T_{i,L}^{-1} (x_{n,j}^{\alpha})\right\}$ to be the zeros set of the TG function with ordered index pair $(\alpha,L)$ (or simply the TGG points set) $\foralla n \in \MBZP, i \in \MBN_2$, where
\begin{equation}\label{eq:transinv1}
{T_{i,L}^{-1}}(x) = \left\{ \begin{array}{l}
L (1+x)/(1-x),\quad i = 1,\\
L  \ln[2/(1-x)],\quad i = 2.
\end{array} \right.
\end{equation}
In particular, we refer to the TGG points sets $\MBTG_{i,n}^{\alpha,L}\,\forall i \in \MBN_2$ by the RG-Gauss (RGG) and EG-Gauss (EGG) points sets, respectively. The following lemma provides the TGG quadrature formula of a function $f \in \SCR{T}_{i,2n+1}^{\alpha,L}$, which is useful for the derivation of the proposed TG-IPS and TG-IS methods.
\begin{lem}\label{lem:1}
The TGG quadrature of a function $f \in \SCR{T}_{i,2n+1}^{\alpha,L}$ is given by
\begin{equation}\label{eq:6}
\C{I}_{\MBRzerP}^{(t)} \left(f\,w^{\alpha,L}_i\right) = {}_i\bmvarpi^{\alpha,L}_n f^{\alpha,L}_{i,n,0:n}\quad \forall (i,n,\alpha,L) \in \MBN_2 \times \MBZzerP \times \MBRmh \times \MBRP,
\end{equation}
where ${}_i\bmvarpi^{\alpha,L}_n = {}_i\varpi^{\alpha,L}_{n,0:n}$: 
\begin{equation}\label{eq:7}
{}_i\varpi^{\alpha,L}_{n,j} = 1/\left[{\lambda^{\alpha}_{0:n}}^{\hspace{-1.3mm}\div} \left(\C{G}_{i,0:n}^{\alpha,L}\cancbra{{}_i^Lt^{\alpha}_{n,j}}\right)_{(2)}\right]\quad \forall (i,j) \in \MBN_2 \times \MBJ_n.  
\end{equation}
\end{lem}

It is easy to show that ${}_i\bmvarpi^{\alpha,L}_n = \bmvarpi^{\alpha}_n\,\forall i \in \MBN_2$, where $\bmvarpi^{\alpha}_n = \varpi^{\alpha}_{n,0:n}$ is the usual GG quadrature weight vector with elements
\begin{equation}
\varpi^{\alpha}_{n,j} = 1/\left[{\lambda^{\alpha}_{0:n}}^{\hspace{-1.3mm}\div} \left(G_{0:n}^{\alpha}\cancbra{x^{\alpha}_{n,j}}\right)_{(2)}\right]\quad \forall j \in \MBJ_n.
\end{equation}
The discrete inner product and the discrete norm associated with the TGG interpolation points are given by
\begin{equation}
(u,v)_{w_i^{\alpha,L},n} = \sum_{j \in \MBJ_n} {u_{i,n,j}^{\alpha,L} v_{i,n,j}^{\alpha,L}\,\varpi_{n,j}^{\alpha}},\quad \left\|u\right\|_{w_i^{\alpha,L},n} = (u,u)_{w_i^{\alpha,L},n}^{1/2}\quad \forall \{u,v\} \subset \Def{\MBRzerP},
\end{equation}
respectively.

\section{TG Interpolation/Collocation and Its Stability}
\label{sec:STGI1}
Using the discrete inner products derived in \ref{sec:TOTIHOCP1}, we can derive the TG  interpolant
\begin{equation}
I_nf(t) = {{\tf}_{0:n}}\,\C{G}_{i,0:n}^{\alpha,L}\cancbra{t},
\end{equation}
that satisfies the interpolation condition
\begin{equation}
I_nf_{i,n,0:n}^{\alpha,L} = f_{i,n,0:n}^{\alpha,L}\quad \forall i \in \MBN_2.
\end{equation}
The orthogonality of the TG basis functions allows us to find the coefficients of the TG interpolant in the following form:
\begin{equation}
\tf_k = \frac{(I_nf,\C{G}_{i,k}^{\alpha,L})_{w_i^{\alpha,L}, n}}{\left\|\C{G}_{i,k}^{\alpha,L}\right\|_{w_i^{\alpha,L},n}^2} = \frac{(f,\C{G}_{i,k}^{\alpha,L})_{w_i^{\alpha,L}, n}}{\left\|\C{G}_{i,k}^{\alpha,L}\right\|_{w_i^{\alpha,L},n}^2}\quad \forall (k,i) \in \MBJ_n \times \MBN_2.
\end{equation}
The following lemma is needed to investigate the stability of TG interpolation. 
\begin{lem}\label{lem:hhl1}
The Christoffel numbers for the Gegenbauer weight function $w^{\alpha}$ are asymptotically bounded by
\begin{equation}\label{eq:Fant1}
\varpi _{n,j}^\alpha \simlteq \varpi^{\text{upp},+} = \frac{\pi }{{n + 1}}\quad \forall (n,j,\alpha) \in \,\MBZPL \times \MBJ_n \times \MBRzerP.
\end{equation}
\end{lem}
\begin{proof}
First notice that 
\begin{subequations}
\begin{gather}
\frac{d}{{d\alpha }}\lambda _j^\alpha  = \frac{{{2^{2\alpha  - 1}}{\Gamma ^2}\left( {\alpha  + \frac{1}{2}} \right)\Gamma \left( {j + 1} \right)}}{{{{\left( {j + \alpha } \right)}^2}\Gamma \left( {j + 2\alpha } \right)}}\left[ {\left( {j + \alpha } \right)\left( {\ln 4 + 2\left( {\digamma\left( {\alpha  + \frac{1}{2}} \right) - \digamma\left( {j + 2\alpha } \right)} \right)} \right) - 1} \right] < 0\quad \forall (j,\alpha) \in \MBZzerP \times \MBRmh,\label{eq:Recall1}\\
\frac{d}{{dj}}\lambda _j^\alpha  = \frac{{{2^{2\alpha  - 1}}\Gamma^2{\left( {\alpha  + \frac{1}{2}} \right)}\Gamma \left( {j + 1} \right)}}{{{{\left( {j + \alpha } \right)}^2}\Gamma \left( {j + 2\alpha } \right)}}\left[ {\left( {j + \alpha } \right)\left( {{H_j} - {H_{j + 2\alpha - 1}}} \right) - 1} \right] < 0\quad \forall (j,\alpha) \in \MBZzerP \times \MBRP,\label{eq:Recall2}
\end{gather}
\end{subequations}
since $\digamma$ and $H_x$ are both strictly increasing with respect to their arguments. Therefore, $\lambda_j^{\alpha}$ is strictly decreasing with respect to its indices $\forall (j,\alpha) \in \MBZzerP \times \MBRP$; cf. Figure \ref{fig:Fig5} for a sketch of $\lambda_j^{\alpha}$ and its reciprocal. This result together with the fact that $\left\|G_k^{\alpha}\right\|_{\infty} \le 1\,\forall \alpha \in \MBRzerP$; cf. \cite[Lemma 2.1]{elgindy2013optimal}, implies
\begin{equation}\label{eq:sof1}
\varpi _{n,j}^\alpha  = 1/\left[{\lambda^{\alpha}_{0:n}}^{\hspace{-1.3mm}\div} \left(G_{0:n}^{\alpha}\cancbra{x^{\alpha}_{n,j}}\right)_{(2)}\right] \simlt \frac{1}{{(n + 1){{\left( {\lambda _0^0} \right)}^{ - 1}}}} = \varpi^{\text{upp},+}\quad \forall (n,j,\alpha ) \in \,\MBZPL \times \MBJ_n \times \MBRP.
\end{equation}
On the other hand, recall that
\begin{equation}
\lambda _j^0  = \left\{ \begin{array}{l}
\pi,\quad j = 0,\\
\pi/2,\quad j \in \MBZP,
\end{array} \right.
\end{equation}
and 
\begin{equation}\label{eq:Sel1}
\varpi_{n,j}^0 = \varpi^{\text{upp},+}\quad \forall j \in \MBZzerP.
\end{equation}
The proof is established by combining the results of \eqref{eq:sof1} and \eqref{eq:Sel1}.
\end{proof}
Figures \ref{fig:Fig6} and \ref{fig:Fig7} show the profiles of $\varpi_{100,j}^{\alpha}$ for $n \in \{100,101\}, j \in \MBJ_n$, and several positive $\alpha$ values. The figures show clearly the symmetrical property of $\varpi_{n,j}^{\alpha}$ about the $j = n/2$ axis for positive $\alpha$ values with maximum values at $j = \left\lfloor {n/2} \right\rfloor\,\forall n \in \MBZP$. Initially, the family of curves appears as parabolas that open downward for $\alpha \in (0,1/2]$ with the axis of symmetry $j = \left\lfloor {n/2} \right\rfloor\,\forall n \in \MBZP$, and then transform into bell-shaped curves for $\alpha > 1/2$ that nearly die out near the boundaries $j = 0, n$ such that $\varpi_{n,j}^{\alpha}$ becomes almost nonzero only on a small subdomain centered at $j = n/2$ for increasing values of $\alpha$. 

\begin{rem}
Lemma \ref{lem:hhl1} agrees with the earlier bounds of the $n$-point GG quadrature weights obtained in \cite{forster1990estimates}, which shows that 
\begin{equation}
\varpi _{n,j}^\alpha  \le \frac{\pi }{{n + \alpha}} \sin^{2 \alpha}\theta_{n,j}^{\alpha}\quad \forall \alpha \in \FOmega_1,
\end{equation}
where $x_{n,j}^{\alpha} = -\cos \theta_{n,j}^{\alpha}$.
\end{rem}

\begin{figure}[H]
\centering
\includegraphics[scale=0.37]{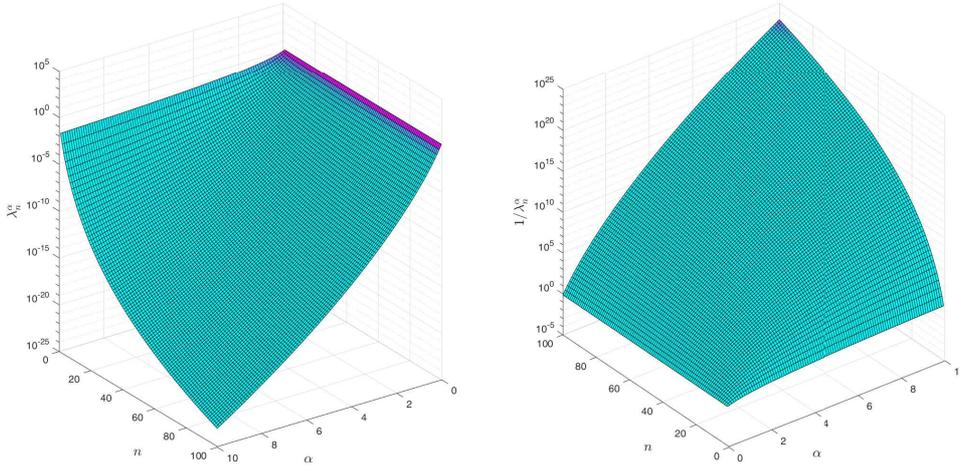}
\caption{The surfaces of $\lambda_n^{\alpha}$ and its reciprocal on the discrete rectangular domain $\{(n,\alpha): n = 1$:$100, \alpha = 0.01$:$0.0999$:$10\}$.}
\label{fig:Fig5}
\end{figure}

\begin{figure}[H]
\centering
\includegraphics[scale=0.37]{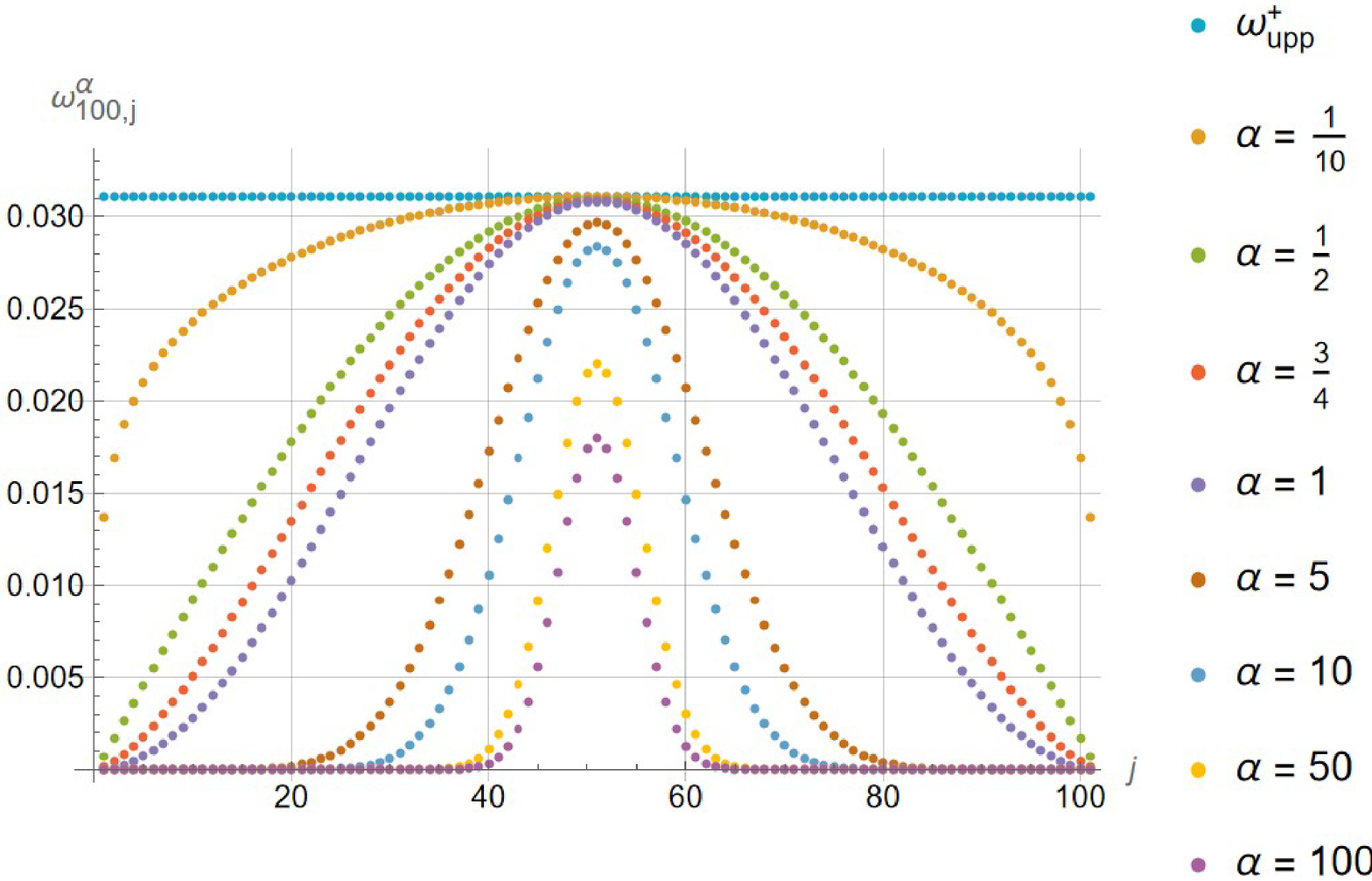}
\caption{Plot of $\varpi_{100,j}^{\alpha}$ for $(j,\alpha) \in \MBJ_{100} \times \{0.1,0.5,1,5,10,50,100\}$.}
\label{fig:Fig6}
\end{figure}

\begin{figure}[H]
\centering
\includegraphics[scale=0.37]{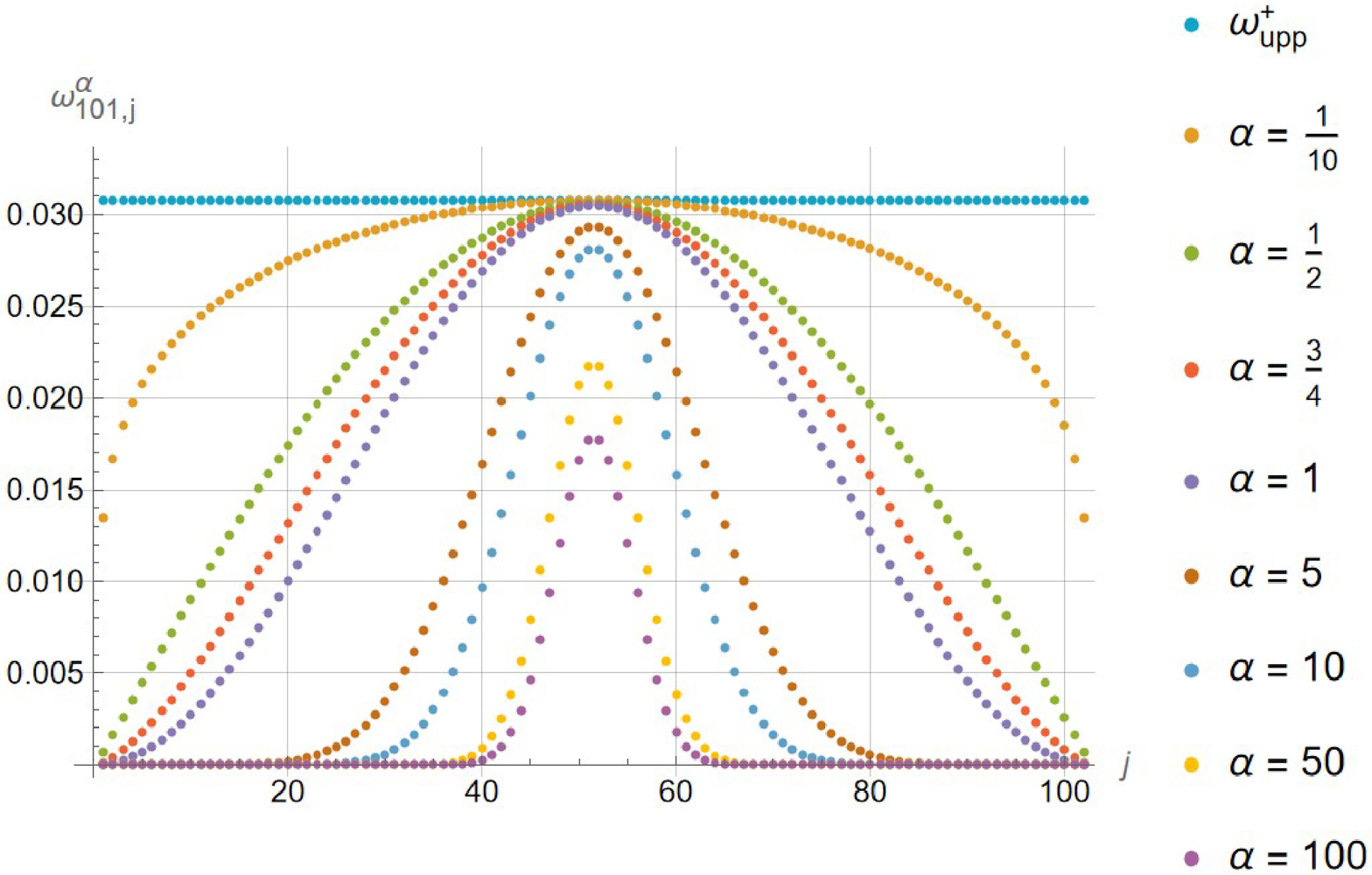}
\caption{Plot of $\varpi_{101,j}^{\alpha}$ for $(j,\alpha) \in \MBJ_{101} \times \{0.1,0.5,1,5,10,50,100\}$.}
\label{fig:Fig7}
\end{figure}

The following two theorems highlight the stability of TG interpolation for $\alpha \in \MBRzerP$ using several space norms.
\begin{thm}\label{thm:hiheq1}
$\foralla u \in L_{w_i^{\alpha ,L}}^2 \left(\MBRzerP\right)$, we have
\begin{equation}
{\left\| {I_n}u \right\|_{L_{w_i^{\alpha ,L}}^2 \left(\MBRzerP\right)}} \simlteq \sqrt{\pi}\,{\left\| u \right\|_{L^{\infty} \left(\MBRzerP\right)}}\quad \forall (n,\alpha,L,i) \in \,\MBZPL \times \MBRzerP \times \MBRP \times \MBN_2.
\end{equation}
\end{thm}
\begin{proof}
The proof is straightforward by using Lemma \ref{lem:hhl1} and realizing that
\begin{equation}\label{eq:bbmbm1}
{\left\| {I_n}u \right\|_{L_{w_i^{\alpha ,L}}^2 \left(\MBRzerP\right)}^2} = 
\left\|I_nu\right\|_{w_i^{\alpha,L},n}^2 = \sum_{j \in \MBJ_n} {\left(u_{i,n,j}^{\alpha}\right)^2 \varpi_{n,j}^{\alpha}} \simlteq \varpi^{\text{upp},+} \sum_{j \in \MBJ_n} {\left(u_{i,n,j}^{\alpha}\right)^2}\quad \forall n \in \,\MBZPL,
\end{equation}
whence the proof is established.
\end{proof}

\begin{thm}\label{thm:hiheq2}
$\foralla u \in H^1_{w_i^{\alpha,L}}\left(\MBRzerP\right)\,\exists\,c \in \MBRP:$
\begin{equation}
{\left\| {I_n}u \right\|_{L_{w_i^{\alpha ,L}}^2 \left(\MBRzerP\right)}} < c \left({\left\| u \right\|_{L_{w_i^{\alpha ,L}}^2 \left(\MBRzerP\right)}} + \frac{1}{n} \left\|\sqrt{(t+1) \sqrt{t}}\,u' \right\|_{L^2\left(\MBRzerP\right)}\right)\quad \forall (n,\alpha,L,i) \in \,\MBZPL \times \MBRzerP \times \MBRP \times \MBN_2.
\end{equation}
\end{thm}
\begin{proof}
Let $t = (1+\cos \theta)/(1-\cos \theta), \hu(\theta) = u\left((1+\cos \theta)/(1-\cos \theta)\right)$, and $\SCR{K}_{n,j} = [2\pi j/(2n+1),(2j+1)\pi/(2n+1)]\,\forall j \in \MBJ_n$. Then the asymptotic inequality \eqref{eq:bbmbm1} yields
\begin{equation}\label{eq:bbmbm2}
{\left\| {I_n}u \right\|_{L_{w_i^{\alpha ,L}}^2 \left(\MBRzerP\right)}^2} \simlteq \varpi^{\text{upp},+} \sum_{j \in \MBJ_n} {\mathop {\sup }\limits_{\theta \in \SCR{K}_{n,j}} \hu^2(\theta)}\quad \forall (n,\alpha,L,i) \in \,\MBZPL \times \MBRzerP \times \MBRP \times \MBN_2.
\end{equation}
The rest of the proof follows the proof of \cite[Theorem 4.1]{guo2002chebyshev}.
\end{proof}
Theorems \ref{thm:hiheq1} and \ref{thm:hiheq2} show that the TG interpolation of a function is asymptotically stable for $\alpha \in \MBRzerP$ under mild conditions. To investigate the stability of TG interpolation for $\alpha \in\,\MBRmhzer$, we shall need the following lemma. 
\begin{lem}\label{lem:wq1}
The Christoffel numbers for the Gegenbauer weight function $w^{\alpha}$ are asymptotically bounded by
\begin{equation}\label{eq:Fant12}
\varpi _{n,j}^\alpha < \varpi^{\text{upp,-}}_{n,\alpha}\quad \forall (n,j,\alpha) \in \,\MBZPL \times \MBJ_n \times \MBRmhzer,
\end{equation}
where
\begin{equation}
\varpi^{\text{upp,-}}_{n,\alpha} = \frac{\Gamma^2(\alpha+1/2)}{2\, n^{1+2\alpha}}.
\end{equation}
\end{lem}
\begin{proof}
The sharp lower bound for the gamma function \cite[Inequality (96)]{Elgindy2016} yields
\begin{equation}
\Gamma (n + 2\alpha ) > \sqrt {2\pi \,{\gamma _n}} {\left( {\frac{{{\gamma _n}}}{e}} \right)^{{\gamma _n}}}{\left[ {{\gamma _n}\sinh \left( {\frac{1}{{{\gamma _n}}}} \right)} \right]^{{\gamma _n}/2}} \sim \sqrt {2\pi \,{\gamma _n}} {\left( {\frac{{{\gamma _n}}}{e}} \right)^{{\gamma _n}}}\quad \forall n \in \,\MBZPL,
\end{equation}
where $\gamma_n = n + 2\alpha -1$. Using Stirling's asymptotic approximation
\begin{equation}
	n! \sim \sqrt{2 \pi n}\,{\left( {\frac{n}{e}} \right)^n}\quad \forall n \in \,\MBZPL,
\end{equation}
we can easily show that 
\begin{equation}
\frac{n!}{\Gamma(n+2\alpha)} \simlt e^{2\alpha-1} \frac{n^{n+1/2}}{\gamma_n^{\gamma_n+1/2}}\quad \forall n \in \,\MBZPL.
\end{equation}
Therefore,
\begin{align}
\lambda_n^{\alpha} \simlt (2e)^{2\alpha-1} \Gamma^2(\alpha+1/2) \frac{n^{n+1/2}}{(n+\alpha)\,\gamma_n^{\gamma_n+1/2}} &< (2e)^{2\alpha-1} \Gamma^2(\alpha+1/2) \frac{n^{n+1/2}}{n^{n+2\alpha+1/2} \left[1+(2\alpha-1)/n\right]^{n+2\alpha+1/2}}\nonumber\\
&\sim 2^{2\alpha-1} \Gamma^2(\alpha+1/2)\,n^{-2\alpha}\quad \forall n \in \,\MBZPL.\label{eq:vv123}
\end{align}
Recall from Ineq. \eqref{eq:Recall1} that $\lambda_j^{\alpha}$ is strictly decreasing with respect to $\alpha\,\forall (j,\alpha) \in \MBZzerP \times \MBRmh$. Formula \eqref{eq:Recall2}, on the other hand, shows that $\displaystyle{\frac{d}{{dj}}}\lambda _j^\alpha > 0\,\forall (j,\alpha) \in \MBZzerP \times \MBRmhzer$, so that $\lambda_j^{\alpha}$ is strictly increasing with respect to $j\,\forall (j,\alpha) \in \MBZzerP \times \MBRmhzer$; cf. Figure \ref{fig:Fig52} for a further graphical evidence. If we add to these results the facts that (i) the GG points move monotonically toward the boundaries of the interval $(-1, 1)$ as the parameter $\alpha$ decreases \cite{Elgindy2016,elgindy2023direct}, (ii) $G_{n}^{\alpha}(\pm 1) = (-1)^n$; cf. \cite[Eq. (A.1)]{elgindy2013optimal}, and (iii) $\left\|G_n^{\alpha}\right\|_{L^{\infty}([-1,1])} \approx D_{\alpha} n^{-\alpha}\,\forall (n,\alpha) \in \left(\MBZPL,\MBRmhzer\right)$:
\begin{equation}
1 < D_{\alpha} = 2^{1-\alpha} \displaystyle{\left| {\frac{{\Gamma (2\alpha )}}{{\Gamma (\alpha )}}} \right|} \to  \infty,\quad \text{as }\alpha \to -1/2,
\end{equation} 
as can be derived from \cite[Eqs. (2.44)--(2.48)]{elgindy2013solving}, then we can reasonably bound $\varpi _{n,j}^\alpha$ by
\begin{equation}\label{eq:sof12}
\varpi _{n,j}^\alpha \simlt 1/\left[{\left(\lambda^{\alpha}_{n}\right)^{-1}} \bmone_{n+1}^{\top} \left(G_{0:n}^{\alpha}\cancbra{x^{\alpha}_{n,n}}\right)_{(2)}\right] \sim \lambda^{\alpha}_{n}/(n+1)\quad \forall (n,j,\alpha ) \in \,\MBZPL \times \MBJ_n \times \MBRmhzer.
\end{equation}
Formulas \eqref{eq:vv123} and \eqref{eq:sof12} therefore yield
\begin{equation}
\varpi _{n,j}^\alpha \simlt 2^{2\alpha-1} \Gamma^2(\alpha+1/2)\,n^{-2\alpha} (n+1)^{-1} < \frac{1}{2} \Gamma^2(\alpha+1/2)\,n^{-1-2\alpha
}\quad \forall (n,j,\alpha ) \in \,\MBZPL \times \MBJ_n \times \MBRmhzer.
\end{equation}
\end{proof}
Figure \ref{fig:Fig9} shows the profile of $\varpi_{100,j}^{\alpha}$ for $j \in \MBJ_{100}$ and several negative $\alpha$ values. We can again observe that the family of curves of $\varpi_{100,j}^{\alpha}$ are parabolas that remain symmetric about the $j = n/2$ axis for negative $\alpha$ values but open upward at a faster rate as $\alpha$ decreases, moving away from its flat shape at $y = \varpi_{\text{upp}}^+$ when $\alpha = 0$; therefore, the maximum values of $\varpi_{n,j}^{\alpha}$ occur at $j \in \{0, n\}\,\forall (n,\alpha) \in \MBZP \times \MBRmhzer$. Figure \ref{fig:Fig10} shows further the curves of $\varpi_{100,j}^{\alpha}$ and their derived upper bounds $\varpi_{100,\alpha}^{\text{upp},-}\,\forall j \in \MBJ_{100}$, and several values of $\alpha$. 

\begin{figure}[H]
\centering
\includegraphics[scale=0.37]{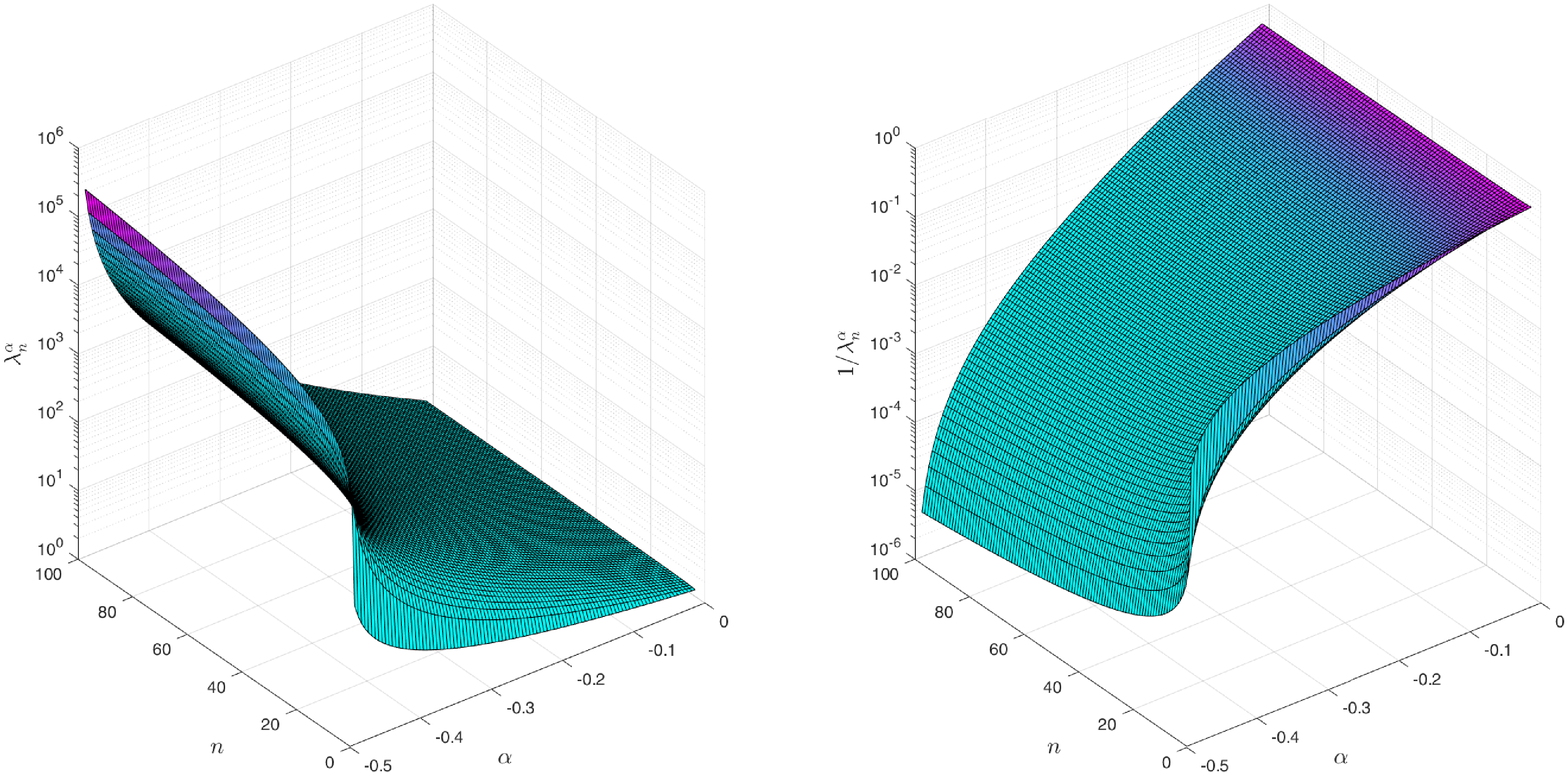}
\caption{The surfaces of $\lambda_n^{\alpha}$ and its reciprocal on the discrete rectangular domain $\{(n,\alpha): n = 1$:$100, \alpha = -0.49$:$0.0048$:$-0.01\}$.}
\label{fig:Fig52}
\end{figure}

\begin{figure}[H]
\centering
\includegraphics[scale=0.37]{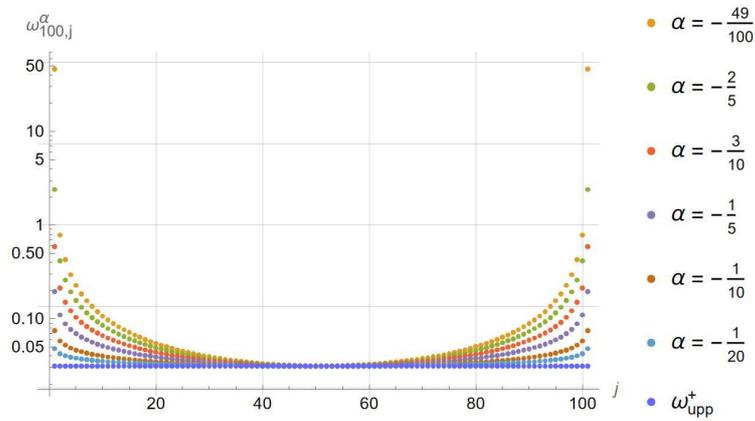}
\caption{Plot of $\varpi_{100,j}^{\alpha}$ for $(j,\alpha) \in \MBJ_{100} \times \{-0.49, -0.4:0.1:-0.1, -0.05\}$.}
\label{fig:Fig9}
\end{figure}

\begin{figure}[H]
\centering
\includegraphics[scale=0.37]{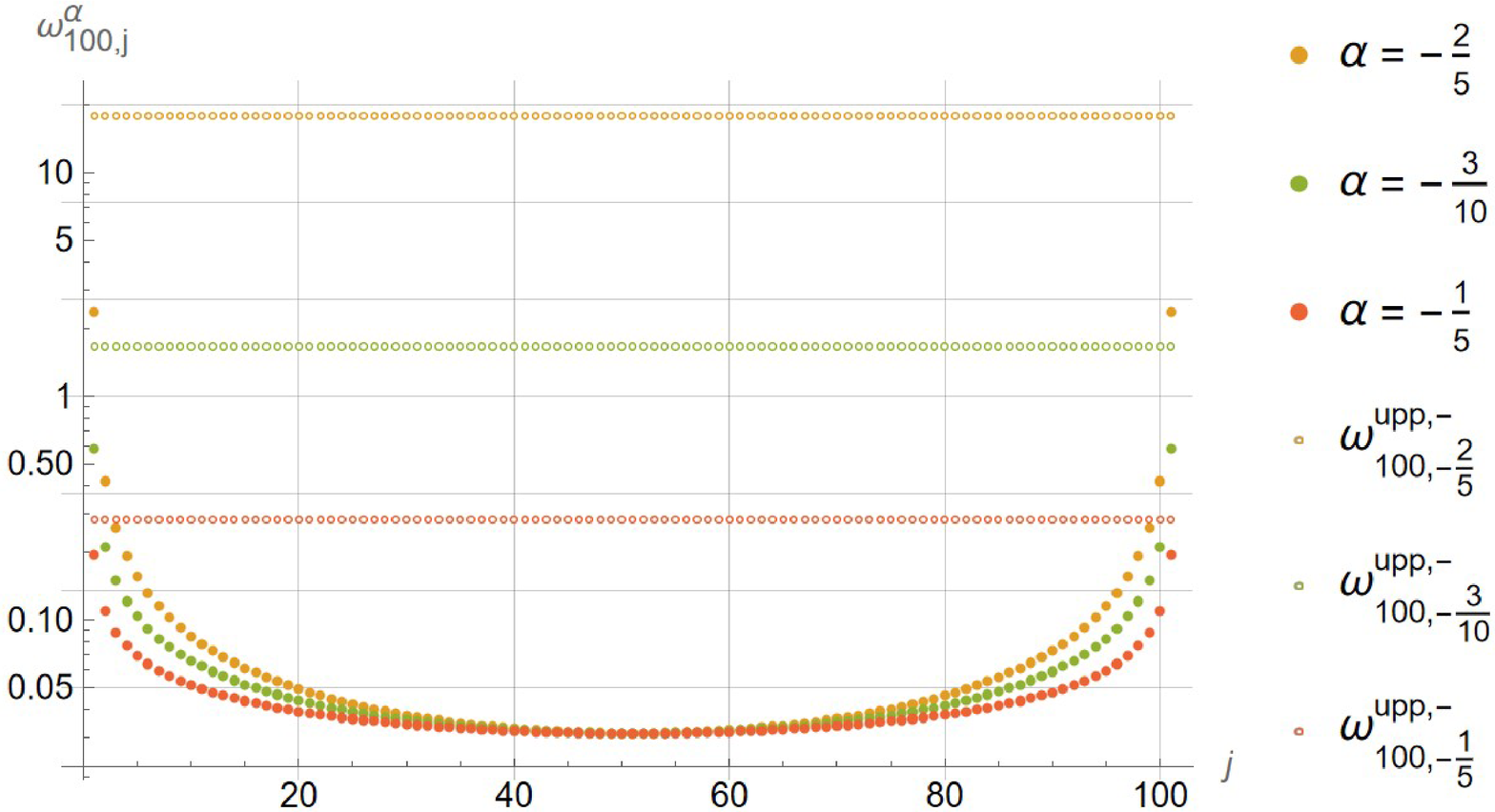}
\caption{Plots of $\varpi_{100,j}^{\alpha}$ and $\varpi_{100,\alpha}^{\text{upp},-}$ for $(j,\alpha) \in \MBJ_{100} \times \{-0.4:0.1:-0.2\}$.}
\label{fig:Fig10}
\end{figure}

The following two theorems highlight the stability of TG interpolation for $\alpha \in\,\MBRmhzer$ using several space norms.
\begin{thm}\label{thm:hiheq1new2}
$\foralla u \in L_{w_i^{\alpha ,L}}^2 \left(\MBRzerP\right)$, we have
\begin{equation}
{\left\| {I_n}u \right\|_{L_{w_i^{\alpha ,L}}^2 \left(\MBRzerP\right)}} < \frac{\Gamma(\alpha+1/2)}{\sqrt{2}\, n^{1/2+\alpha}}\,{\left\| u \right\|_{L^{\infty} \left(\MBRzerP\right)}}\quad \forall (n,\alpha,L,i) \in \,\MBZPL \times \MBRmhzer \times \MBRP \times \MBN_2.
\end{equation}
\end{thm}
\begin{proof}
The proof follows readily by Formula \eqref{eq:bbmbm1} and Lemma \ref{lem:wq1}.
\end{proof}

\begin{thm}\label{thm:hiheq222}
$\foralla u \in H^1_{w_i^{\alpha,L}}\left(\MBRzerP\right)\,\exists\,c \in \MBRP:$
\begin{equation}
{\left\| {I_n}u \right\|_{L_{w_i^{\alpha ,L}}^2 \left(\MBRzerP\right)}} < c\, \Gamma(\alpha+1/2) \sqrt{\frac{n+1}{n^{1+2\alpha}}} \left({\left\| u \right\|_{L_{w_i^{\alpha ,L}}^2 \left(\MBRzerP\right)}} + \frac{1}{n} \left\|\sqrt{(t+1) \sqrt{t}}\,u' \right\|_{L^2\left(\MBRzerP\right)}\right)\quad \forall (n,\alpha,L,i) \in \,\MBZPL \times \MBRmhzer \times \MBRP \times \MBN_2.
\end{equation}
\end{thm}
\begin{proof}
By closely following the proof of Theorem \ref{thm:hiheq2}, we have 
\begin{align}\label{eq:bbmbm232}
{\left\| {I_n}u \right\|_{L_{w_i^{\alpha ,L}}^2 \left(\MBRzerP\right)}^2} &< \varpi_{n,\alpha}^{\text{upp},-} \sum_{j \in \MBJ_n} {\mathop {\sup }\limits_{\theta \in \SCR{K}_{n,j}} \hu^2(\theta)}\quad \forall (n,\alpha,L,i) \in \,\MBZPL \times \MBRzerP \times \MBRP \times \MBN_2.\\
\Rightarrow \frac{2 \pi n^{1+2\alpha}}{(n+1) \Gamma^2(\alpha+1/2)} {\left\| {I_n}u \right\|_{L_{w_i^{\alpha ,L}}^2 \left(\MBRzerP\right)}^2} &< \frac{\pi}{n+1} \sum_{j \in \MBJ_n} {\mathop {\sup }\limits_{\theta \in \SCR{K}_{n,j}} \hu^2(\theta)}\\
&< c_1 \left({\left\| u \right\|^2_{L_{w_i^{\alpha ,L}}^2 \left(\MBRzerP\right)}} + \frac{1}{n^2} \left\|\sqrt{(t+1) \sqrt{t}}\,u' \right\|^2_{L^2\left(\MBRzerP\right)}\right)\quad \foralls c_1 \in \MBRP,
\end{align}
by following the proof of \cite[Theorem 4.1]{guo2002chebyshev}. Therefore,
\begin{equation}
\frac{1}{\Gamma(\alpha+1/2)} \sqrt{\frac{2 \pi n^{1+2\alpha}}{(n+1)}} {\left\| {I_n}u \right\|_{L_{w_i^{\alpha ,L}}^2 \left(\MBRzerP\right)}} < c_2 \left({\left\| u \right\|_{L_{w_i^{\alpha ,L}}^2 \left(\MBRzerP\right)}} + \frac{1}{n} \left\|\sqrt{(t+1) \sqrt{t}}\,u' \right\|_{L^2\left(\MBRzerP\right)}\right)\quad \foralls c_2 \in \MBRP,
\end{equation}
whence the proof is established.
\end{proof}
Theorems \ref{thm:hiheq1new2} and \ref{thm:hiheq222} show that the TG interpolation is asymptotically unstable for $(n,\alpha) \in\,\MBZPL \times \MBRmh$, and the instability increases as $\alpha \to -1/2$.

\section*{Declarations}
\subsection*{Competing Interests}
The authors declare there is no conflict of interest.

\subsection*{Availability of Supporting Data}
The authors declare that the data supporting the findings of this study are available within the article.

\subsection*{Ethical Approval and Consent to Participate and Publish}
Not Applicable.

\subsection*{Human and Animal Ethics}
Not Applicable.

\subsection*{Consent for Publication}
Not Applicable.

\subsection*{Funding}
The authors received no financial support for the research, authorship, and/or publication of this article.

\subsection*{Acknowledgment}
Special thanks to Dr. Jihyoung Cha\footnote{Dr. Jihyoung Cha is a researcher at the Department of Computer Science, Electrical and Space Engineering, Lule\aa\, University of Technology, Sweden.} for sharing the initial conditions and calculated optimal cost function values of Problems 1 and 2 used in \cite{cha2019infinite}, which were not explicitly reported there.

\bibliographystyle{elsarticle-num-names}
\bibliography{Bib}

\begin{thebibliography}{21}
\expandafter\ifx\csname natexlab\endcsname\relax\def\natexlab#1{#1}\fi
\providecommand{\url}[1]{\texttt{#1}}
\providecommand{\href}[2]{#2}
\providecommand{\path}[1]{#1}
\providecommand{\DOIprefix}{doi:}
\providecommand{\ArXivprefix}{arXiv:}
\providecommand{\URLprefix}{URL: }
\providecommand{\Pubmedprefix}{pmid:}
\providecommand{\doi}[1]{\href{http://dx.doi.org/#1}{\path{#1}}}
\providecommand{\Pubmed}[1]{\href{pmid:#1}{\path{#1}}}
\providecommand{\bibinfo}[2]{#2}
\ifx\xfnm\relax \def\xfnm[#1]{\unskip,\space#1}\fi
\bibitem[{Ko and Bitmead(2007)}]{KO20071573}
\bibinfo{author}{S.~Ko}, \bibinfo{author}{R.~R. Bitmead},
\newblock \bibinfo{title}{Optimal control for linear systems with state
  equality constraints},
\newblock \bibinfo{journal}{Automatica} \bibinfo{volume}{43}
  (\bibinfo{year}{2007}) \bibinfo{pages}{1573--1582}.
\bibitem[{Cha et~al.(2019)Cha, Kang, and Ko}]{cha2019infinite}
\bibinfo{author}{J.~Cha}, \bibinfo{author}{S.~Kang}, \bibinfo{author}{S.~Ko},
\newblock \bibinfo{title}{Infinite horizon optimal output feedback control for
  linear systems with state equality constraints},
\newblock \bibinfo{journal}{International Journal of Aeronautical and Space
  Sciences} \bibinfo{volume}{20} (\bibinfo{year}{2019})
  \bibinfo{pages}{483--492}.
\bibitem[{Li et~al.(2020)Li, Lai, Jin, and Yu}]{li2020diagonalized}
\bibinfo{author}{S.~Li}, \bibinfo{author}{Z.~Lai}, \bibinfo{author}{L.~Jin},
  \bibinfo{author}{X.~Yu},
\newblock \bibinfo{title}{Diagonalized {G}egenbauer rational spectral methods
  for second-and fourth-order problems on the whole line},
\newblock \bibinfo{journal}{Applied Numerical Mathematics}
  \bibinfo{volume}{151} (\bibinfo{year}{2020}) \bibinfo{pages}{494--516}.
\bibitem[{Hajimohammadi et~al.(2020)Hajimohammadi, Baharifard, and
  Parand}]{hajimohammadi2020new}
\bibinfo{author}{Z.~Hajimohammadi}, \bibinfo{author}{F.~Baharifard},
  \bibinfo{author}{K.~Parand},
\newblock \bibinfo{title}{A new numerical learning approach to solve general
  {F}alkner--{S}kan model},
\newblock \bibinfo{journal}{Engineering with Computers}  (\bibinfo{year}{2020})
  \bibinfo{pages}{1--17}.
\bibitem[{Baharifard et~al.(2022)Baharifard, Parand, and
  Rashidi}]{baharifard2022novel}
\bibinfo{author}{F.~Baharifard}, \bibinfo{author}{K.~Parand},
  \bibinfo{author}{M.~Rashidi},
\newblock \bibinfo{title}{Novel solution for heat and mass transfer of a {MHD}
  micropolar fluid flow on a moving plate with suction and injection},
\newblock \bibinfo{journal}{Engineering with Computers}  (\bibinfo{year}{2022})
  \bibinfo{pages}{1--18}.
\bibitem[{Parand et~al.(2018)Parand, Bahramnezhad, and
  Farahani}]{parand2018numerical}
\bibinfo{author}{K.~Parand}, \bibinfo{author}{A.~Bahramnezhad},
  \bibinfo{author}{H.~Farahani},
\newblock \bibinfo{title}{A numerical method based on rational {G}egenbauer
  functions for solving boundary layer flow of a {P}owell--{E}yring
  non-{N}ewtonian fluid},
\newblock \bibinfo{journal}{Computational and Applied Mathematics}
  \bibinfo{volume}{37} (\bibinfo{year}{2018}) \bibinfo{pages}{6053--6075}.
\bibitem[{Elgindy and Refat(2023)}]{elgindy2023direct}
\bibinfo{author}{K.~T. Elgindy}, \bibinfo{author}{H.~M. Refat},
\newblock \bibinfo{title}{A direct integral pseudospectral method for solving a
  class of infinite-horizon optimal control problems using {G}egenbauer
  polynomials and certain parametric maps},
\newblock \bibinfo{journal}{AIMS Mathematics} \bibinfo{volume}{8}
  (\bibinfo{year}{2023}) \bibinfo{pages}{3561--3605}.
\bibitem[{Dahy and Elgindy(2022)}]{dahy2022high}
\bibinfo{author}{S.~A. Dahy}, \bibinfo{author}{K.~T. Elgindy},
\newblock \bibinfo{title}{High-order numerical solution of viscous {B}urgers'
  equation using an extended {C}ole--{H}opf barycentric {G}egenbauer integral
  pseudospectral method},
\newblock \bibinfo{journal}{International Journal of Computer Mathematics}
  \bibinfo{volume}{99} (\bibinfo{year}{2022}) \bibinfo{pages}{446--464}.
\bibitem[{Elgindy(2018)}]{Elgindy2016}
\bibinfo{author}{K.~T. Elgindy},
\newblock \bibinfo{title}{Optimal control of a parabolic distributed parameter
  system using a fully exponentially convergent barycentric shifted
  {G}egenbauer integral pseudospectral method},
\newblock \bibinfo{journal}{Journal of Industrial and Management Optimization}
  \bibinfo{volume}{14} (\bibinfo{year}{2018}) \bibinfo{pages}{473--496}.
\bibitem[{Elgindy and Dahy(2018)}]{elgindy2018high}
\bibinfo{author}{K.~T. Elgindy}, \bibinfo{author}{S.~A. Dahy},
\newblock \bibinfo{title}{High-order numerical solution of viscous {B}urgers'
  equation using a {C}ole-{H}opf barycentric {G}egenbauer integral
  pseudospectral method},
\newblock \bibinfo{journal}{Mathematical Methods in the Applied Sciences}
  \bibinfo{volume}{41} (\bibinfo{year}{2018}) \bibinfo{pages}{6226--6251}.
\bibitem[{Dai et~al.(2016)Dai, Yang, He et~al.}]{dai2016integral}
\bibinfo{author}{M.~Dai}, \bibinfo{author}{X.~Yang}, \bibinfo{author}{Y.~He},
  et~al.,
\newblock \bibinfo{title}{Integral form and equivalence proof of three
  pseudospectral optimal control methods},
\newblock \bibinfo{journal}{Control and Decision} \bibinfo{volume}{6}
  (\bibinfo{year}{2016}) \bibinfo{pages}{1123--1127}.
\bibitem[{Greengard(1991)}]{greengard1991spectral}
\bibinfo{author}{L.~Greengard},
\newblock \bibinfo{title}{Spectral integration and two-point boundary value
  problems},
\newblock \bibinfo{journal}{SIAM Journal on Numerical Analysis}
  \bibinfo{volume}{28} (\bibinfo{year}{1991}) \bibinfo{pages}{1071--1080}.
\bibitem[{Driscoll(2010)}]{driscoll2010automatic}
\bibinfo{author}{T.~A. Driscoll},
\newblock \bibinfo{title}{Automatic spectral collocation for integral,
  integro-differential, and integrally reformulated differential equations},
\newblock \bibinfo{journal}{Journal of Computational Physics}
  \bibinfo{volume}{229} (\bibinfo{year}{2010}) \bibinfo{pages}{5980--5998}.
\bibitem[{Elgindy(2023{\natexlab{a}})}]{elgindy2023new}
\bibinfo{author}{K.~T. Elgindy},
\newblock \bibinfo{title}{New optimal periodic control policy for the optimal
  periodic performance of a chemostat using a {F}ourier--{G}egenbauer-based
  predictor-corrector method},
\newblock \bibinfo{journal}{Journal of Process Control} \bibinfo{volume}{127}
  (\bibinfo{year}{2023}{\natexlab{a}}) \bibinfo{pages}{102995}.
\bibitem[{Elgindy(2023{\natexlab{b}})}]{elgindy2023fouriera}
\bibinfo{author}{K.~T. Elgindy},
\newblock \bibinfo{title}{{F}ourier-{G}egenbauer pseudospectral method for
  solving periodic fractional optimal control problems},
\newblock \bibinfo{journal}{arXiv preprint arXiv:2304.04454}
  (\bibinfo{year}{2023}{\natexlab{b}}).
\bibitem[{Elgindy(2023{\natexlab{c}})}]{elgindy2023fourierb}
\bibinfo{author}{K.~T. Elgindy},
\newblock \bibinfo{title}{{F}ourier-{G}egenbauer pseudospectral method for
  solving time-dependent one-dimensional fractional partial differential
  equations with variable coefficients and periodic solutions},
\newblock \bibinfo{journal}{arXiv preprint arXiv:2304.14061}
  (\bibinfo{year}{2023}{\natexlab{c}}).
\bibitem[{Elgindy and Smith-Miles(2013{\natexlab{a}})}]{elgindy2013fast}
\bibinfo{author}{K.~T. Elgindy}, \bibinfo{author}{K.~A. Smith-Miles},
\newblock \bibinfo{title}{Fast, accurate, and small-scale direct trajectory
  optimization using a {G}egenbauer transcription method},
\newblock \bibinfo{journal}{Journal of Computational and Applied Mathematics}
  \bibinfo{volume}{251} (\bibinfo{year}{2013}{\natexlab{a}})
  \bibinfo{pages}{93--116}.
\bibitem[{Elgindy and Smith-Miles(2013{\natexlab{b}})}]{elgindy2013optimal}
\bibinfo{author}{K.~T. Elgindy}, \bibinfo{author}{K.~A. Smith-Miles},
\newblock \bibinfo{title}{Optimal {G}egenbauer quadrature over arbitrary
  integration nodes},
\newblock \bibinfo{journal}{Journal of Computational and Applied Mathematics}
  \bibinfo{volume}{242} (\bibinfo{year}{2013}{\natexlab{b}})
  \bibinfo{pages}{82--106}.
\bibitem[{F{\"o}rster and Petras(1990)}]{forster1990estimates}
\bibinfo{author}{K.-J. F{\"o}rster}, \bibinfo{author}{K.~Petras},
\newblock \bibinfo{title}{On estimates for the weights in {G}aussian quadrature
  in the ultraspherical case},
\newblock \bibinfo{journal}{Mathematics of computation} \bibinfo{volume}{55}
  (\bibinfo{year}{1990}) \bibinfo{pages}{243--264}.
\bibitem[{Guo et~al.(2002)Guo, Shen, and Wang}]{guo2002chebyshev}
\bibinfo{author}{B.-y. Guo}, \bibinfo{author}{J.~Shen}, \bibinfo{author}{Z.-q.
  Wang},
\newblock \bibinfo{title}{Chebyshev rational spectral and pseudospectral
  methods on a semi-infinite interval},
\newblock \bibinfo{journal}{International journal for numerical methods in
  engineering} \bibinfo{volume}{53} (\bibinfo{year}{2002})
  \bibinfo{pages}{65--84}.
\bibitem[{Elgindy and Smith-Miles(2013)}]{elgindy2013solving}
\bibinfo{author}{K.~T. Elgindy}, \bibinfo{author}{K.~A. Smith-Miles},
\newblock \bibinfo{title}{Solving boundary value problems, integral, and
  integro-differential equations using {G}egenbauer integration matrices},
\newblock \bibinfo{journal}{Journal of Computational and Applied Mathematics}
  \bibinfo{volume}{237} (\bibinfo{year}{2013}) \bibinfo{pages}{307--325}.

\end{thebibliography}
\end{document}